\documentclass[12pt]{amsart}
\usepackage{amssymb,amsthm,amsmath,amsfonts}
\usepackage{hyperref}
\usepackage{enumerate}

\usepackage{pgfplots}
\pgfplotsset{width=8cm,compat=1.9}
\usepgfplotslibrary{fillbetween}
\usetikzlibrary{patterns}

\makeatletter
\g@addto@macro\th@plain{\thm@headpunct{}}
\makeatother

\newtheorem{theorem}{Theorem}[section]
\newtheorem{proposition}[theorem]{Proposition}
\newtheorem{Corollary}[theorem]{Corollary}
\newtheorem{lemma}[theorem]{Lemma}
\newtheorem{remark}[theorem]{Remark}

\def \R {{\mathbb R}}
\def\N{{\mathbb{N}}}
\def\Z{{\mathbb{Z}}}
\def \P {{\mathbb P}}
\def \E {{\mathbb E}}
\def \Q {{\mathbb Q}}
\def \H {{\mathbf H}}
\def \V {{\mathbf V}}

\newcommand{\dd}{\mathrm{d}}
\newcommand{\supp}{\mathrm{supp}}
\newcommand{\Ind}{\mathbf{1}_}

\newcommand{\eps}{\varepsilon}
\newcommand{\g}{\gamma }
\newcommand{\bt}{\beta^{\ast}}
\newcommand{\ba}{\tilde{\beta}}
\renewcommand{\a}{\alpha}
\renewcommand{\b}{\beta}
\renewcommand{\d}{\delta}

\title[Stochastic recursions]{Stochastic recursions: between Kesten's and Grincevi\v{c}ius-Grey's assumptions}
\author[E. Damek]{Ewa Damek}
\address{Institute of Mathematics\\ Wroclaw University\\ 50-384 Wroclaw\\
pl. Grunwaldzki 2/4\\ Poland}
\email{edamek@math.uni.wroc.pl}

\author[B. Ko\l{}odziejek]{Bartosz Ko\l{}odziejek}
\address{Faculty of Mathematics and Information Science\\Warsaw University of Technology\\ Koszykowa 75\\00-662 Warsaw, Poland}
\email{b.kolodziejek@mini.pw.edu.pl}



\begin{document}

\begin{abstract}
We study the stochastic recursion $X_n=\Psi_n(X_{n-1})$, where $(\Psi_n)_{n\geq 1}$ is a sequence of i.i.d. random Lipschitz mappings close to the random affine transformation $x\mapsto Ax+B$. We describe the tail behaviour of the stationary solution $X$ under the assumption that 
there exists $\alpha>0$ such that $\E |A|^{\a}=1$ and the tail of $B$ is regularly varying with index $-\a<0$. We also find the second order asymptotics of the tail of $X$ when $\Psi(x)=Ax+B$.
\end{abstract}

\maketitle
\section{Introduction}
\subsection{Results and motivation} Let $(\Psi_n)_{n\geq 1}$ be a sequence of i.i.d (independent identically distributed) random Lipschitz real mappings. Given $X_0$ independent of $(\Psi_n)_{n\geq 1}$ we study stochastic recursions
\begin{equation}\label{affine}
X_n=\Psi_n(X_{n-1}), \quad n\geq 1
\end{equation} 
known also as iterated function systems (IFS). Beginning from the early nineties IFS modeled on Lipschitz functions attracted a lot of attention \cite{Alsm2014,AC,BB,DF,Du,Elton,HH,Mi}.
Under mild contractivity hypotheses, $X_n$ converges in law to a random variable $X$ satisfying (in distribution)
\begin{equation}\label{affine1}
X\stackrel{d}{=}\Psi(X),\qquad X\mbox{ and }\Psi\mbox{ are independent},
\end{equation} 
where $\Psi$ is a generic element of the sequence $(\Psi_n)_{n\geq 1}$ \cite{DF,Elton}. However, to describe the tail of $X$ some further assumptions are needed. Usually one assumes that $\Psi $ is close to an affine mapping or, more precisely, that  for every $x\in\R$
\begin{align}\label{ineqPsi1}
Ax+B_1\leq \Psi(x)\leq Ax+B,\qquad \mbox{a.s.}
\end{align}
with $A$, $B$ and $B_1$ nice enough.
The reason is that if $\Psi (x)=Ax+B$, then the tail of stationary distribution is thoroughly described under various assumptions on $A$ and $B$, see Section \ref{previous}.

 In the present paper we consider two kinds of approximations: \eqref{ineqPsi1} and the case when
$\Psi\colon \R\to \R$ is a random {Lipschitz} mapping satisfying for all $x\geq0$
\begin{align}\label{ineqPsi}
\max\{Ax,B_1\}\leq \Psi(x)\leq Ax+B,\qquad \mbox{a.s.}
\end{align}
Under suitable conditions on $A$, $B$ and $B_1$ we obtain asymptotics of $\P(X>x)$ as $x\to \infty$ in both cases, see Theorems \ref{main1} and \ref{main2}.

There is a number of the papers on the subject \cite{Alsm2014,BD,DD,Dy,Mi}, where the IFS are modeled on the assumptions needed to handle the tail in the affine recursion. Typical conditions exhibit existence of certain moments of $A$ and $B$ or regular behaviour of their tails and in all the settings considered up to now either $A$ or $B$ has  basically the ultimate influence on the tail, not both. A short overview is given in Section \ref{previous}.

We study an opposite situation. For the time being, we assume $A\geq 0$ a.s., $\E A^{\alpha}=1$ and $B, B_1$ have right tails regularly varying with index $-\alpha$ for some $\alpha>0$ such that $\E |B|^{\a }, \E |B_1|^{\a }$ are infinite\footnote{If $\P (B>x)\sim x^{-\alpha }L(x)$ then $\E |B|^{\alpha }$ may be finite or infinite depending on the slowly varying function $L$}. Our starting point is the tail behaviour of $X_{\max}$ being the stationary solution to ``so called'' extremal recursion,  corresponding to $\Psi (x)=\max \{Ax,B\}$. Then  
\begin{align}\label{eq:max}
x^\alpha \P(X_{\max}>x)\sim\frac{1}{\rho}\int_{0}^x\frac{L(t)}{t}\dd t \quad \mbox{as}\ x\to\infty,
\end{align}
see \cite{DK18}, where $L(x)=x^{\a}\P (B>x)$ is assumed to be slowly varying function, $\sim$ is defined in \eqref{sim} and $\rho$ is as in {\rm(A-2)}\footnote{Note that there is no issue with integrability of $L(t)/t$ near $0+$ because $L(t)\leq t^\alpha$ for $t>0$.}.
More precisely, if conditions {\rm(A-1)}, {\rm(A-2)}, {\rm(B-1)}, {\rm(AB-1)} defined in Theorem \ref{perpet}
hold then \eqref{eq:max} follows and the right hand side of \eqref{eq:max} is due to both the behaviour of $B$ and of an appropriate renewal measure determined by $A$. Moreover, $\int_{0}^xL(t)t^{-1}\dd t/L(x)$ tends to infinity as $x\to\infty$ and $x\mapsto \int_{0}^xL(t) t^{-1}\dd t$ is again slowly varying.

The next step is to prove a result in the spirit of \eqref{eq:max} for
$\Psi (x)=Ax+B$, see Theorem \ref{perpet} below.  While the behaviour of the right tails of stationary distribution of the extremal and the affine recursion turn out to be the same, the asymptotics 
$$
x^{\alpha }\P (X>x)\sim \frac{1}{\rho}\int_{0}^x\frac{L(t)}{t}\dd t$$
of $X$ corresponding to \eqref{ineqPsi} is a straight forward conclusion, Theorem \ref{main1} in Section \ref{lipschitz}. Neither the affine recursion nor iterated function systems have been considered under these assumptions and the appearance of the function 
\begin{equation}\label{tilde}
\widetilde{L}(x)=\int_{0}^x\frac{L(t)}{t}\dd t \end{equation} 
is probably the most interesting phenomenon here. For the IFS satisfying \eqref{ineqPsi1} we prove that both $\P (X>x)$ and $\P (X<-x)$ have similar behaviour for large $x$, Theorem \ref{main2}.

\subsection{Perpetuities} Before we formulate precisely the results for Lipschitz iterations let us discuss solutions to the affine recursion with  $\Psi (x)=Ax+B$. Such solutions, if exist, are called \emph{perpetuities} and throughout the paper
they will be denoted by $R$. It exists and it is unique if $\E \log|A| <0$ and $\E \log ^+|B|<\infty $ which is guaranteed by assumptions of Theorem \ref{perpet}. For two functions $f,g$ we write 
\begin{equation}\label{sim}
f(x)\sim g(x)\qquad\mbox{if} \qquad \lim _{x\to \infty} f(x)\slash g(x)=1.\end{equation}
Recall that $L$ is slowly varying if $L(x)\sim L(\lambda x)$ for any $\lambda>0$. Let $B_+=\max\{B,0\}$ and $B_-=\max\{-B,0\}$. We have the following theorem
\begin{theorem}\label{perpet}
Suppose that 
\begin{itemize}
	\item[\rm (A-1)] $A\geq 0$ a.s. and the law of $\log A$ given $A>0$ is non-arithmetic,
	\item[\rm (A-2)] there exists $\alpha>0$ such that 
				 $\E A^{\alpha}=1$, $\rho:=\E A^{\alpha}\log A<\infty$,
  \item[\rm (B-1)] $L(x):=x^\alpha\P(B>x)$ is slowly varying,  $\E B_+^\alpha=\infty$ and $\E B_-^{\alpha-\eps}<\infty$ for all $\eps\in(0,\alpha)$,
			\item[\rm (AB-1
)] $\E A^{\eta}B_+^{\alpha-\eta}<\infty$ for some $\eta\in(0,\alpha)\cap (0,1]$.
	\end{itemize}
Then
\begin{align}\label{eq:objectiver}
x^\alpha \P(R>x)\sim\frac{1}{\rho}\int_{0}^x\frac{L(t)}{t}\dd t\quad  %\to\infty.
\end{align}
\end{theorem}

We see that the behaviour of $\P(R>x)$ as $x\to\infty$ is described in terms of the behaviour of the tail of $B_+$. Accordingly, the behaviour of $P(R<-x)$ depends on the tail of $B_-$. To see this, let us denote $B_1=-B$. Then, $R_1=-R$ satisfies
\[
R_1\stackrel{d}{=}A R_1+B_1,\qquad R_1\mbox{ and }(A,B_1)\mbox{ are independent.}
\]
and the right tail of $R_1$ is the same as the left tail of $R$. We thus obtain the following result.

\begin{Corollary}
	Assume {\rm(A-1)} and {\rm(A-2)} and   
	\begin{itemize}
		\item[{\rm (B-2)}] $L_1(x):=x^\alpha\P(B<-x)$ is slowly varying, $\E B_-^\alpha=\infty$ and $\E B_+^{\alpha-\eps}<\infty$ for all $\eps\in(0,\alpha)$,
		\item[{\rm (AB-2)}] $\E A^{\eta}B_-^{\alpha-\eta}<\infty$ for some $\eta\in(0,\alpha)\cap (0,1]$,
	\end{itemize}
	then
	\begin{align}\label{eq:objectivel}
	x^\alpha \P(R<-x)\sim\frac{1}{\rho}\int_{0}^x\frac{L_1(t)}{t}\dd t.
	\end{align}
	Finally, if all the above assumptions and additionally {\rm (B-1)}, {\rm (AB-1)} are satisfied, then we have both \eqref{eq:objectiver} and \eqref{eq:objectivel} with possibly different slowly varying functions $L$ and $L_1$.
\end{Corollary}

To obtain tail asymptotics one usually applies an appropriate renewal theorem and so do we. However, what we need goes beyond existing results and we prove a new one, Theorem \ref{LTH}. 
Note that under {\rm(A-1) and} {\rm(A-2)}
$$\rho=\E A^{\alpha}\log A$$ 
 is strictly positive.
Indeed, consider $f(\beta):=\E A^\b$. Since $f(0)=1=f(\alpha)$, $f$ is convex, we have $f'(\a)=\rho>0$.
Secondly, observe that, under $\E |B|^{\alpha -\eta }<\infty $, \eqref{eq:objectiver} depends only on the regular behaviour of the right tail of $B_+$ and so we may obtain different asymptotics for $\P (R>x)$ and $\P (R<-x)$ if it is so for $B$. It follows from \eqref{Bax} that 
\[
\int_{0}^x\frac{L(t)}{t}\dd t \sim \frac{1}{\alpha} \E B_+^\alpha \Ind{B\leq x}.
\]
and so the right hand sides of \eqref{eq:objectiver} and
\eqref{eq:objectivel} tend to $\infty $ when $x\to \infty$.
Finally, conditions {\rm(AB-1)} and {\rm(AB-2)} require a comment. If $\E A^{\alpha +\eps }<\infty $ for some $\eps >0$ then they both are satisfied by H\"older inequality. But much less is needed. Namely, if $\E A^{\alpha }W(A)<\infty $, where $W(x)=\max \{\widetilde{L}(x),\log x\}$ then 
{\rm(AB-1)} and \rm {(AB-2)} hold,
see the Appendix.

Next we study the second order asymptotics of the right tail of $R$. Assuming more regularity of $\log A$, we prove that
\begin{equation}\label{eq:second}
\Big | x^{\a }\P (R>x)-\frac{1}{\rho }\int _0^x\frac{L(t)}{t}\dd t-C\Big |= O(L(x))+o(1),\quad \mbox{as}\ x\to \infty, 
\end{equation}
for some constant $C$; see Theorem \ref{perp2}. Notice that either $L(x)$ or $1$ may dominate the right hand side of \eqref{eq:second}. \eqref{eq:second} holds  
%The above result holds %under rather strong assumptions, which necessitated the %introduction of the class of measures $\mu$, for which the corresponding 
when the renewal measure determined by $\log A$ 
%$\H=\sum_{n=0}^\infty \mu^{\ast n}$ 
satisfies
\[
\H((x,x+h])\leq c\max\{h^\beta,h\}
\]
for some $\beta\in(0,1)$ and for all $x, h \geq 0$ - see Lemma \ref{FROST}.
In view of \cite{Gol91} and \cite{BDP} it is not much of a surprise that stronger assumptions on $\H $ are needed to describe the second order asymptotics of the tail of a perpetuity.

Finally, we develop a new approach to deal with signed $A$. We show how to reduce ``signed $A$'' to ``non-negative $A$'' (see Theorem \ref{signedB} (i)) and we apply our result to the case when $\E|A|^\a=1$. The method is quite general and it is applicable beyond our particular assumptions.

\subsection{Previous results on perpetuities}\label{previous}
$\P(R>x)$ converges to zero when $x$ tends to infinity and a natural problem consists of describing the rate at which this happens. Depending on the assumptions on $(A,B)$ we may obtain light-tailed $R$ (all the moments exist) or a heavy tailed $R$ (certain moments of $|R|$ are infinite). The first case occurs when $\P(|A|\leq 1)=1$ and $B$ has the moment generating function in some neighbourhood of the origin, see \cite{Root,GG96, HitWes09,BKc16, BDIM, BK18}. 

The second case is when 
$\P (|A|>1)>0$ with $\E \log |A|<0$ and $|A|, |B|$ have some positive moments.
Then the tail behaviour of $R$ may be determined by $A$ or $B$ alone, or by both of them. The first case happens when the tail of $B$ is regularly varying with index $-\a<0$, $\E |A|^{\a }<1$ and $\E |A|^{\a+\eps}<\infty$ for some $\eps>0$. Then %if additionally $A\geq 0$,   
\begin{equation}\label{grey}
\P (R>x)\sim c\,\P (B>x),\end{equation}
see \cite{Gri75, Gre94}. 
Also it may happen that 
\[
\P (R>x)\sim c\,\P (A>x)
\]
when $\E |A|^{\a }<1$ but $\P (|B|>x)=O(\P (A>x))$, see \cite{DD}. 
%, but positivity of $A$ is not really essential, see \cite{BDMi2012}  
 When $\E |A|^{\a }=1$, $\E |B|^{\a }<\infty $, $\E |A|^{a }\log ^+ |A|<\infty$ and the law of $\log |A|$ given $\{A\neq 0\}$ is non-arithmetic, then \cite{Kes73,Gri75,Gol91}
\begin{equation}\label{gol}
\P (R>x)\sim c\,x^{-\a }\end{equation}
and it is $A$ that plays the main role. When $\E |A|^{a }\log ^+ |A|=\infty$ an extra slowly varying function $l$ appears in \eqref{gol}, i.e.
\begin{equation}\label{kevei}
\P (R>x)\sim c\, l(x) x^{-\a }.\end{equation} 
\eqref{kevei} was proved by \cite{Kevei2016} for $A\geq 0$ but applying our approach to signed $A$ (see Section \ref{general}) we may conclude \eqref{kevei} also there \footnote{For the results in the case when $\max\{|A|,|B|\}$ does not have positive moments we refer to \cite{Dy}.}.

In view of all that it is natural to go a step further and to ask what happens when at the same time $A$ and $B$ contribute significantly to the tail in the sense of {\rm(A-2)} and {\rm(B-1)}.

\subsection{Lipschitz iterations}\label{lipschitz}
In this section we state the results for IFS and we show how do they follow from \eqref{eq:max} and Theorem \ref{perpet}. We assume that $\Psi $ satisfies conditions sufficient for existence of stationary solution. Let $L(\Psi )$, $L(\Psi _{n,1})$ be the Lipschitz constants of $\Psi $, $\Psi _{n,1}=\Psi _n\circ \dots \circ \Psi _1$ respectively.
If $\E \log ^+L(\Psi )<\infty $, $\E \log ^+ | \Psi (0)|<\infty $  and
\begin{equation}\label{contr}
\lim _{n\to \infty }\frac{1}{n}\log L(\Psi _{n,1})<0 \quad a.s.
\end{equation}
then
 $X_n$ converges in distribution to a random variable $X$, which does not depend on $X_0$ and satisfies \eqref{affine1}.

 For slowly varying functions $L_r$ and $L_{1,r}$ let us denote
	\[
	\widetilde{L}_r(x) = \int_0^x \frac{L_r(t)}{t}\dd t\qquad\mbox{and}\qquad \widetilde{L}_{1,r}(x) = \int_0^x \frac{L_{1,r}(t)}{t}\dd t.
	\]
\begin{theorem}\label{main1}
%	Suppose that {\rm(A-1)}, {\rm(A-2)}, {\rm(B-1)} and {\rm(AB-1)} are satisfied and
Suppose that {\rm(A-1)}, {\rm(A-2)}, {\rm(B-1)} and {\rm(AB-1)} are satisfied both for $B$ and $B_1$ with $L=L_{r}$ and $L=L_{1,r}$ respectively. Let $\Psi$ be such that
	\begin{equation}\label{maxaff}
	\max\{Ax,B_1\}\leq \Psi(x)\leq Ax+B,\quad \mbox{a.s.},\quad x\geq0.
	\end{equation}

Then for every $\eps >0$ and $x$ sufficiently large
\begin{align}\label{eq:objective3}
\frac{1-\eps }{\rho} \widetilde{L}_{1,r}(x)\leq  x^\alpha \P(X>x)\leq \frac{1+\eps} {\rho}\widetilde{L}_r(x).
\end{align}
Particularly, if $\widetilde{L}_r(x)\sim \widetilde{L}_{1,r}(x)$ then   
\begin{equation}\label{eq:objective4}
x^\alpha \P(X>x)\sim \frac{1}{\rho}\widetilde{L}_r(x).
\end{equation}

%	\begin{align}\label{eq:objective}
%	x^\alpha \P(X>x)\sim\frac{1}{\rho}\widetilde{L}_(x)\to\infty,  
%	\end{align}
%	where $\widetilde{L}$ is defined in \eqref{tilde}.
\end{theorem}

\begin{theorem}\label{main2}
If a function $\Psi$ satisfies 
\begin{equation}\label{sandwich}
Ax+B_1\leq \Psi(x)\leq Ax+B,\quad \mbox{a.s.},\quad x\in\R,
\end{equation}
then under the assumptions of Theorem \eqref{main1}, assertions \eqref{eq:objective3} and \eqref{eq:objective4} hold true.

If {\rm(A-1)}, {\rm(A-2)}, {\rm(B-2)} and {\rm(AB-2)} hold both for $B$ and $B_1$ then we have analogous conclusions for $\P (X<-x)$.
\end{theorem}
Theorems \ref{main1} and \ref{main2} follow quickly from Theorem \ref{perpet} and \eqref{eq:max} (i.e. Theorem 4.2 of \cite{DK18}). To see this let us consider Theorem \ref{main2}. Let
\[
R_n=A_nR_{n-1}+B_n,\quad R_{n,1}=A_nR_{n-1,1}+B_{n,1}
\]
with $R_0=R_{0,1}=X_0$. Then for every $n$,
\[
R_{n,1}\leq X_n\leq R_n\quad \mbox{a.s.}
\]
and so
\[
R_{1}\leq X\leq R,\quad\mbox{a.s}
\] 
where $R\stackrel{d}{=}AR+B$ with $(A,B)$ independent of $R$ and similarly for $(A_1,B_1,R_1)$. Hence
\[
x^{\a }\P (X>x) / \widetilde{L}_r(x) \leq x^{\a }\P (R>x) / \widetilde{L}_r(x).
\]
Letting $x\to \infty $, we obtain 
\[ 
\limsup _{x\to \infty}x^{\a }\P (X>x)/\widetilde{L}_r(x)\leq \lim _{x\to \infty}x^{\a }\P (R>x)/\widetilde{L}_r(x)=\frac{1}{\rho},
\]
which implies the right hand side of \eqref{eq:objective3}. The left hand side is obtained analogously. Clearly $\widetilde{L}_r(x)\sim \widetilde{L}_{1,r}(x)$ implies \eqref{eq:objective4}. In the same way we proceed for the proof of Theorem \ref{main1}.

Let us comment on stochastic iterations that fall under the assumption of Theorem \eqref{main1}. Subtracting $A x$ from \eqref{maxaff} we arrive at
\[ 
(B_1 - A x)_+ \leq \Theta(x) \leq B, \qquad\mbox{a.s.,}\qquad x\geq 0,
\]
where we have defined $\Theta(x) : = \Psi(x) - A x$. Analyzing this condition geometrically,  we see that $\Psi(x) = A x + \Theta(x)$ satisfies \eqref{maxaff} if for each $x\geq0$, the value of $\Theta(x)$ belongs (a.s.) to patterned part of the figure below.  If $\Theta(x) \in [B_1,B]$ a.s. for each $x$, then $\Psi$ satisfies \eqref{sandwich}. Moreover, $\Theta$ may be chosen in the way that $\E \log L(\Theta)<0$ which implies \eqref{contr}.

\begin{center}
	\begin{tikzpicture}
	\begin{axis}[
	axis x line = left,
	axis y line = left,
	xlabel = {$x$},
	xticklabels = {0,0, ,$-B_1/A$},
	yticklabels = {0, 0,, , $B_1$, $B\,$},
%	xlabel pos = right,
	ymin = 0,
	ymax = 86,
	xmin = 0,
	xmax = 6,
	major tick length = 0,
%	xlabel absolute,
%	xlabel near ticks
	]
	
	\addplot [
	name path = A,
	domain=0:6, 
	samples=100, 
	color=red, 
	thick
	]
	{80};
	
	\addplot [
	name path = B,
	domain=0:6, 
	samples=100, 
	color=blue,
	thick
	]
	{max(0,60 - x*30)};
	
	\addplot [gray, pattern = north east lines, pattern color = gray!50] fill between [ of=B and A];
	\end{axis}
	\end{tikzpicture}
\end{center}

\subsection{Structure of the paper}
Theorem \ref{perpet} is proved in Section \ref{taila}. Section \ref{second}. is devoted to the second order asymptotics. 
Before, we need some preliminaries on the renewal theory. A renewal theorem which is the basic tool is formulated in Section \ref{reneval} and proved in the last section. Subsection \ref{extrareg} contains material needed only for the second order asymptotics.
We deal with general $A$ in Section \ref{general}.
 
\section{Preliminaries}\label{Preli}

\subsection{Regular variation}
A measurable function $L\colon(0,\infty)\to(0,\infty)$ is called \emph{slowly varying}, (denoted $L\in R(0)$), 
if for all $\lambda>0$,
\begin{align}\label{reg}
\lim_{x\to\infty}\frac{L(\lambda x)}{L(x)}=1.
\end{align}
For $\rho\in\R$ we write $R(\rho)$ for the class of \emph{regularly varying functions with index $\rho$}, which consists of functions $f$ of the form $f(x)=x^\rho L(x)$ for some $L\in R(0)$. 

If $L\in R(0)$ is bounded away from $0$ and $\infty$ on every compact subset of $[0,\infty)$, then for any $\delta>0$ there exists $A=A(\delta)>1$ such that (Potter's Theorem, see e.g \cite{BDM2016}, Appendix B)
\begin{equation}\label{potter}
\frac{L(y)}{L(x)}\leq A \max\left\{\left(\frac yx\right)^{\delta},\left(\frac xy\right)^{\delta}\right\},\qquad x,y>0.\end{equation}

Assume that $L\in R(0)$ is locally bounded on $(x_0,\infty)$ for some $x_0>0$. Then, for $\alpha>0$ one has
\begin{align}\label{Lreg}
\int_{x_0}^x t^\alpha \frac{L(t)}{t}\mathrm{d}t\sim \alpha^{-1}x^\alpha L(x)
\intertext{and this result remains true also for $\alpha=0$ in the sense that}
\frac{\int_{x_0}^x \frac{L(t)}{t}\mathrm{d}t}{L(x)}\to\infty\qquad\mbox{ as }x\to\infty,\label{Lreg2}
\end{align}
\cite[Proposition 1.5.9a]{BGT89}. Define $\widetilde{L}_{x_0}(x):=\int_{x_0}^x t^{-1}L(t)\dd t$. Function $\widetilde{L}_{x_0}$ is sometimes called de Haan function. It is again slowly varying and has the property that for any $\lambda>0$,
\begin{align}\label{smallint}
\frac{\widetilde{L}_{x_0}(\lambda x)-\widetilde{L}_{x_0}(x)}{L(x)}=\int_1^{\lambda}\frac{L(xt)}{L(x)}\frac{\dd t}{t}\to\log\lambda,
\end{align}
To prove it, use the fact that convergence in \eqref{reg} is locally uniform \cite[Theorem 1.5.2]{BGT89}.

\subsection{Renewal theory}
Let $(Z_k)_{k\geq 1}$ be the sequence of independent copies of random variable $Z$ with $\E Z>0$. 
We write $S_n=Z_1+\ldots+Z_n$ for $n\in\mathbb{N}$ and $S_0=0$.
The measure $\H$ defined on Borel sets $\mathcal{B}(\R)$ by 
$$\H(B):=\sum_{n=0}^\infty \P(S_n\in B),\qquad B\in\mathcal{B}(\R)$$
is called the \emph{renewal measure of $(S_n)_{n\geq 1}$},
$H(x):=\H((-\infty,x])$ is called the \emph{renewal function}. If $\E Z>0$, then $H(x)$ is finite for all $x\in\R$ if and only if $\E Z_-^2<\infty$ (\cite{KM96}).

We say that the distribution of $Z$ is \emph{arithmetic} if its support is contained in $d\Z$ for some $d>0$; otherwise it is \emph{non-arithmetic}. Equivalently, the distribution of $Z$ is arithmetic if and only if there exists $0\neq t\in\R$ such that $f_Z(t)=1$, where $f_Z$ is the characteristic function of the distribution of $Z$.
The law of $Z$ is \emph{strongly non-lattice} if the Cramer's condition is satisfied, that is, $\limsup_{|t|\to\infty}|f_Z(t)|<1$.

A fundamental result of renewal theory is the Blackwell theorem (see \cite{Black53}): if the distribution of $Z$ is non-arithmetic, then for any $h>0$, 
\[\lim_{x\to \infty }\H((x,x+h])\to \frac{h}{\E Z}.\]
Note that in the non-arithmetic case, since $\H((x,x+h])$ is convergent as $x\to\infty$ we have \mbox{$C=\sup_{x}\H((x,x+1])<\infty$} and so
\begin{align}\label{Hineq1}
\H\left((x,x+h]\right)\leq \left\lceil h\right\rceil C\leq \alpha h+\beta ,\quad \mbox{for}\ x\in \R .
\end{align}
for some positive $\alpha$, $\beta$ and any $h>0$. 

Under additional assumptions we know more about the asymptotic behaviour of $\H$ and $H$ (see \cite{Stone65}).
If for some $r>0$ one has $\P(Z\leq x)=o(e^{rx})$ as $x\to-\infty$, then there is some $r_1>0$ such that 
\begin{align}\label{Hleft}
H(x)=o(e^{r_1 x})\qquad\mbox{ as }x\to-\infty.
\end{align}
More accurate asymptotics of $H(x)$ as $x\to-\infty$ is given in \cite{BK16}. If $Z$ has finite second moment, for some $r>0$, $\P(Z>x)=o(e^{-r x})$ as $x\to\infty$ and the distribution of $Z$ is strongly non-lattice, then there is  $r_1>0$ such that (see \cite{Stone65})
\begin{align}\label{Hright}
H(x)=\frac{x}{\E Z}+\frac{\E Z^2}{2(\E Z)^2}+o(e^{-r_1 x})\qquad \mbox{ as }x\to\infty.
\end{align}

\subsection{Renewal measure with extra regularity}\label{extrareg} For the second order asymptotics we need a better control of 
$\H\left((x,x+h]\right)$ in terms of $h$ than \eqref{Hineq1}; something in the spirit of 
\begin{align}\label{Hfrost}
\H\left((x,x+h]\right)\leq c\, h^\beta, \qquad x\geq0,h>0
\end{align}
for some $\beta>0$.
Observe that with $C_n=\sup_x\H((x,x+1/n])<\infty$ we have
$$\H\left((x,x+h]\right)\leq C_n \frac{\left\lceil n h\right\rceil}{n}$$
thus \eqref{Hfrost} holds for all $x$ and $h>1/n$ with $\beta=1$. 
Hence, we have to investigate the case of small $h$ only.
We have the following statement.
\begin{lemma}\label{FROST}
	Assume that $\P(Z>x)=o(e^{-r x})$ as $x\to\infty$ for some $r>0$, $\E Z_-^2<\infty$ and that the law of $Z$ is strongly non-lattice.
	If there exists $\beta>0$ such that
	\begin{align}\label{cond}
	\limsup_{h\to 0^+}\sup_{a\geq 0}h^{-\beta}\P(a<Z\leq a+h)<\infty,
	\end{align}
	then there exists $\tilde\beta>0$ and $c>0$ such that for $x\geq 0$ and $h\geq0$, 
	\begin{align}\label{Hfrostp} \H\left((x,x+h]\right)\leq c \max \{ h^{\tilde\beta}, h\}.\end{align}
\end{lemma}
\begin{remark}
Notice that \eqref{cond} is satisfied when the law of $Z$ has density in $L^p$ for some $1< p \leq \infty$. 
\end{remark}
Before we write the proof let us describe a certain factorization of $\H$ that will be used in it. 
In renewal theory it is usually easier to consider first a non-negative $Z$, and then to extend some argument to arbitrary $Z$ using the following approach (see e.g. proof of Lemma \ref{FROST}).
Let \mbox{$N=\inf\{ n\in\N \colon S_n>0\}$} be the first ladder epoch of $(S_n)_{n\geq1}$. We define a measure by 
\[\V(B):=\E\left(\sum_{n=0}^{N-1} \Ind{S_n\in B}\right),\qquad B\in\mathcal{B}(\R).\]
The support of $\V$ is contained in $(-\infty,0]$ and $\V(\R)=\E N$. Since $(S_n)_{n\geq 1}$ has a positive drift, $\E N$ is finite.
Let $Z_1^>\stackrel{d}{=}S_N$ be the first ladder height of $(S_n)_{n\geq 1}$ and consider an i.i.d. sequence $(Z_n^>)_{n\geq 1}$.
Then%, it can be shown that
\[\H=\V\ast \H^{>},\]
where $\H^>$ is the renewal measure of $(S_n^>)_{n\geq1}$ and $S_n^>=\sum_{k=1}^n Z_k^{>}$  (\cite[Theorem 2]{Black53}, see also \cite[Lemma 2.63]{Als} for more general formulation).  

\begin{proof}[Proof of Lemma~\ref{FROST}]
We will first consider the case when $Z\geq 0$ a.s. Let $F$ be the cumulative distribution function of $Z$.
From condition \eqref{cond} we infer that there exists $\beta,c,\varepsilon>0$
	such that for any $a\geq0$ and any $h\in(0,\eps]$ one has $F(a+h)-F(a)=\P(a<Z\leq a+h)\leq c h^\beta$. Decreasing $\eps$, if needed, we can and do assume that $F(\eps)<1$.
Since $H(x)=\Ind{x\geq 0}+H\ast F(x)$ we have for any $x\geq 0$ and $h\in(0,\eps]$,
\begin{align*}
\H\left((x,x+h]\right)&=\int_{[0,x]}\left(F(x-z+h)-F(x-z)\right)\H(\dd z)+\int_{(x,x+h]} F(x+h-z)\H(\dd z) \\
& \leq c h^\beta \H([0,x])+F(h) \H\left((x,x+h]\right)
\end{align*}
and thus 
\[
\H\left((x,x+h]\right)\leq (1-F(\eps))^{-1} c\, h^\beta \H([0,x]).
\]

Let now $Z$ be arbitrary and let $S_N$ be the first ladder height of $(S_n)_{n\geq1}$. Since $\E N<\infty$ and, for $a\geq 0$ and small enough $h$,
\begin{align*}
\P(a<S_N\leq a+h)=\sum_{n=1}^\infty \P(a<S_n\leq a+h, S_1\leq 0,\ldots,S_{n-1}\leq 0, S_n>0)  \\
 = \sum_{n=1}^\infty \P(a-S_{n-1}<Z_n\leq a-S_{n-1}+h, N\geq n) \leq c h^\beta\sum_{n=1}^\infty\P(N\geq n),
\end{align*}
by \eqref{cond} and it follows that
\[\limsup_{h\to 0^+}\sup_{a\geq 0}h^{-\beta}\P(a<S_N\leq a+h)<\infty.\]
Thus, using factorization $\H=\V\ast \H^>$ we obtain for $x\geq 0$ and $h\in(0,\eps]$,
\[\H\left((x,x+h]\right)=\int_{(-\infty,0]} \H^>((x-t,x-t+h]) \V(\dd t) \leq c h^\beta \int_{(-\infty,0]} \H^>([0,x-t]) \V(\dd t)=c h^\beta H(x).\]
For $0\leq x\leq h^{-\delta}$ with $\delta<\beta$ this implies that
$$\H\left((x,x+h]\right)\leq C h^\beta(1+x)\leq \tilde{C} h^{\beta-\delta}.$$
On the other hand, for $x> h^{-\delta}$ and $r>0$ we have
$$e^{-rx}\leq e^{-rh^{-\delta}}\leq r^{-1} h^\delta,$$
where we have used the fact that $x \exp(-x)<1$ for $x>0$.
The conclusion follows by \eqref{Hright}, since then
$$\H\left((x,x+h]\right)=\frac{h}{\E Z}+o(e^{-rx}).$$
\end{proof}

\section{Renewal Theorem}\label{reneval}
A function $f\colon\R\to\R_+$ is called \emph{directly Riemann integrable} on $\R$ (dRi) if for any $h>0$,
\begin{align}\label{dRi}
\sum_{n\in\Z}\sup_{(n-1)h\leq y<nh}f(y)<\infty
\end{align}
and
$$\lim_{h\to0^+}h\cdot\left(\sum_{n\in\Z}\sup_{(n-1)h\leq y<nh}f(y)-\sum_{n\in\Z}\inf_{(n-1)h\leq y<nh}f(y) \right)=0.$$
If $f$ is locally bounded and a.e. continuous on $\R$, then an elementary calculation shows that \eqref{dRi} with $h=1$ implies direct Riemann integrability of $f$.
For directly Riemann integrable function $f$, we have the following {\it Key Renewal Theorem} (\cite{AMN78}):
$$\int_{\R} f(x-z)\H(\dd z)\to\frac{1}{\E Z}\int_{\R} f(t)\dd t.$$
There are many variants of this theorem, when $f$ is not necessarily $L_1$ - see \cite[Section 6.2.3]{Iks}. Such results are usually obtained under the additional requirement that $f$ is (ultimately) monotone or $f$ is asymptotically equivalent to a monotone function. 

Neither of them is sufficient for us. To prove Theorem \ref{perpet} we need to integrate the function $e^{\a x}\E g(e^{-x}B)$ with respect to $\H$, where $g$ is  $C^1$ function ``approximating'' $\Ind{(1,\infty)}$. Therefore, we prove the following result.

\begin{theorem}\label{LTH}
	Assume that $0<\E Z<\infty$, the law of $Z$ is non-arithmetic and $\P(Z\leq x)=o(e^{rx})$ as $x\to-\infty$. Assume further that there is a random variable $B$ and a slowly varying function $L$ such that $\P(B>x)=x^{-\alpha}L(x)$.
    Let	$g$ be a bounded function,  $\supp\,g\subset [1,\infty)$ and there exists a constant $c$ such that
    	\begin{equation}\label{gcond}
    	\Big | \frac{\dd}{\dd t}\big (e^{-\a t} g(e^{t})\big )\Big |\leq ce^{-\a t},\quad t>0.
    	\end{equation} 
    
	Then 
	\begin{align}\label{eq1}\lim _{x\to \infty}\widetilde L(e^x)^{-1}\int _{\R }e^{\a (x-z)}\E g(e^{z-x}B)\H(\dd z)=	\a (\E Z)^{-1}\int _{[1,\infty)}g(r)r^{-\a -1}\dd r.\end{align}
Assume additionally that $\E \exp(\eps Z)<\infty$ for some $\eps>0$ and that the law of $Z$ is strongly non-lattice. Then as $x\to\infty$,
	\begin{align}\label{eq0}
	\left |\int _{\R }e^{\a (x-z)}\E g(e^{z-x}B)\H(\dd z)-\a \widetilde L(e^x)\int _{[1,\infty)}g(r)r^{-\a -1}\dd r\right |\leq CL(e^x),\end{align}
where $C $ depends on $\| g\| _{sup}  \sup_{x\in \R}|g(x)|$ and the constant $c$ in \eqref{gcond}.  $C\to \infty $ if either $\| g\| _{sup}\to \infty$ or $c\to \infty$.
\end{theorem}
To obtain asymptotics of $\P (R>e^x)$ one {may} integrate $e^{\alpha x}\P (B>e^x)$ with respect to $\H$ {and control other components} as it is explained in the proof of Theorem \ref{function}. However, using $C^1$ functions $g$ instead of {$\Ind{(1,\infty)}(x)$} allows us to {avoid many technical obstacles, without requiring stronger regularity of $\H$}.
Basically we need $g$ as defined in \eqref{defg} i.e. approximating $\Ind{(1,\infty)}(x)$.
Observe that when $g(x)$ is replaced  by $\Ind{(1,\infty)}(x)$ we obtain $e^{\a (x-z)}\E \Ind{(1,\infty)}(e^{z-x}B) = L(e^{x-z})$ and so Theorem 
\ref{LTH} is in analogy to Theorems 3.1 and 3.3 in  \cite{DK18} which say that
\begin{align*}\int_{\R} L\left(e^{x-z}\right)\H(\dd z)\sim\frac{1}{\E Z}\widetilde{L}\left(e^{x}\right)
\end{align*}
or with more regularity on $Z$,
\begin{align*}\int_{\R} L\left(e^{x-z}\right)\H(\dd z)=\frac{1}{\E Z}\widetilde{L}\left(e^{x}\right)+O\left(L(e^x)\right).\end{align*}
The proof of Theorem \ref{LTH} is postponed to the last section.

\section{Perpetuities}

%%%%%%%%%%%%%%%%%%%%%%%%%%%%%%%%%%%%%%%%%%%%%%%%%%%%%%%%
%%%%%%%%%%%%%%%%%%%%%%%%%%%%%%%%%%%%%%%%%%%%%%%%%%%%%%%%
%%%%%%%%%%%%%%%%%%%%%%%%%%%%%%%%%%%%%%%%%%%%%%%%%%%%%%%%
\subsection{First order asymptotics}\label{taila}
%In this section we consider the random affine equation
%\begin{equation}\label{affeq}
%R\stackrel{d}{=}AR+B,\qquad \mbox{where}\quad (A,B)\mbox{ and }R\mbox{ are %independent}.\end{equation}
In this section we prove Theorem \ref{perpet}. 
%We are going to describe the first order asymptotics of right tails of $R$ %under 
The assumptions are the same as in \cite[Theorem 4.2 {\rm(i)}]{DK18}, where the extremal recursion was considered. The proof, however, is not that simple. %because of technical difficulties. 
Therefore, we use a different approach, introduced in \cite{BDGHU2009}. 
Instead of proving directly the asymptotics of $\P(R>x)$ we look for the asymptotics of $\E g(R/x)$, where $g$ is a $C^1$ function and $\supp\,g\subset [1,\infty)$. The advantage of such approach is that certain function is easily shown to be dRi (see Proposition \ref{der}). Moreover, the asymptotics of $\P(R>x)$ follows straightforward from the asymptotics of $\E g(R/x)$ and the whole proof is quite simple.

Theorem \ref{perpet} is an immediate consequence of \eqref{lim} below.
\begin{theorem}\label{function}
	Suppose that 
conditions {\rm(A-1)}, {\rm(A-2)}, {\rm(B-1)}, {\rm(AB-1)} are satisfied. 
Let $g$ be a bounded function supported in $[1,\infty )$. Suppose that \eqref{gcond} holds. Then 
	\begin{equation}\label{lim}
	\lim_{x\to \infty }\frac{x^{\a } \E g(x^{-1}R)}{\widetilde{L}(x)}= \frac{\a}{\rho}\int _{[1,\infty)}g(r)r^{-\a -1}\dd r.
	\end{equation}
Moreover, as $x\to\infty$
\begin{equation}\label{gsecond}
\left |x^{\a } \E g(x^{-1}R) - \frac{\a}{\rho}\widetilde{L}(x)\int _{[1,\infty)}g(r)r^{-\a -1}\dd r \right |\leq C\max \{1,L(x)\}, 
\end{equation}
where $C$ depends on $\| g\| _{sup}$ and $c$ in \eqref{gcond}.
\end{theorem}
Using H\"older {continuous} or $C^1$ functions approximating indicators instead of indicators themselves is a standard procedure which usually allows to reduce regularity {requirement for} the probability distribution in question. By regularity we mean here assumptions similar to \eqref{condA} or even existence of density. They seem to be needed, if indicators are used, but with H\"older continuous functions one can handle calculations differently. In various problems this approach is very successful. 

Although we use regularity of functions in intermediate steps, what we obtain at the end allows us to take the limit and to eliminate the dependence on H\"older constants or derivatives, see e.g Sections 3.1, 3.2 or Appendix D in \cite{BDM2016}. This can be done in \eqref{lim} because the right hand side depends only on the integral of $g$. However this is not the case in \eqref{gsecond} because $C\to \infty $ if $\| g'\| _{sup}\to \infty $ which takes place when indicators are approximated by $C^1$ functions. Therefore, for the second order asymptotics we have to proceed differently.   
The problem is treated in the next Section.

Finally, $\alpha \widetilde{L}(x)$ in \eqref{lim} and \eqref{gsecond} may be replaced by $\E B_+^\alpha \Ind{B\leq x}$.  
As an easy consequence of \eqref{Lreg} we obtain
\begin{proposition}\label{Prop1}
Assume that the first condition in {\rm(B-1)} holds. Then, we have
\begin{align}\label{Bax}
\E B_+^\alpha \Ind{B\leq x}=\alpha \widetilde{L}(x)-L(x)\sim \alpha\widetilde{L}(x)
\end{align}
and for $r>0$, 
\begin{align*}
\E B_+^{\alpha+r}\Ind{B\leq x}&=(\alpha+r)\int_0^x t^{\alpha+r-1}\P(B>t)\dd t-x^{\alpha+r}\P(B>x)\sim {\frac{\alpha}{r}}x^r L(x).
\end{align*}

\noindent Assuming additionally that the second condition in {\rm(B-1)} holds we have  $\widetilde{L}(x)\uparrow\infty$ as $x\to\infty$.
\end{proposition}

\begin{proof}[Proof of Theorem \ref{perpet}]
	It is enough to prove that for a $\xi >1$
	\begin{equation}
	\lim _{x\to \infty }x^{\a}\widetilde L(x)^{-1}\P (R>x\xi )=\rho^{-1}\xi ^{-\a }.
	\end{equation}
	Let $\xi >1$ and $\eta >0$ be such that $\xi -\eta >1$. 
	Let $g_1$ be a $C^1$ function such that $0\leq g_1\leq 1$ and  
	\begin{equation}\label{defg}
	g_1(x)=\begin{cases} 0\ \mbox{if}\ x\leq \xi -\eta \\ 
	1\ \mbox{if}\ x\geq \xi  \end{cases},
	\end{equation}
	Let $g_2(x)=g_1(x-\eta)$. 
	
	Then $g_1, g_2$ satisfy the assumptions of Theorem \ref{function},
 because $g'_1(x)=g'_2(x)=0 $ for $x\leq \xi -\eta $ and $x\geq \xi +\eta$. We have
	\begin{multline*}
	I_2:=\lim _{x\to \infty }x^{\a}\widetilde L(x)^{-1}\E g_2(x^{-1}R)\\
	\leq \liminf _{x\to \infty }x^{\a}\widetilde L(x)^{-1}\P (R>x\xi )
	\leq \limsup _{x\to \infty }x^{\a}\widetilde L(x)^{-1}\P (R>x\xi )\\
	\leq \lim _{x\to \infty }x^{\a }\widetilde L(x)^{-1}\E g_1(x^{-1}R)=:I_1.
	\end{multline*}
	Moreover, 
	\begin{align*}
	|I_1-I_2|&\leq \frac{\a}{\rho} \int _0^{\infty }|g_1(r)-g_2(r)|r^{-\a -1}\dd r \\
	&\leq \frac{\a}{\rho} \int _{\xi -\eta }^{\xi +\eta }r^{-\a -1}\dd r
	\leq   2\a \eta/\rho .
	\end{align*}
Letting $\eta \to 0$ we infer that 
	\begin{equation*}
	\lim _{x\to \infty }x^{\a}\widetilde L(x)^{-1}\P (R>x\xi ) \quad \mbox{exists}. 
	\end{equation*}
Finally,	%and for every $\eta\in(0,1-\eps)$
	\begin{equation*}
	\left|\lim _{x\to \infty }x^{\a }\widetilde L(x)^{-1}\P (R>x\xi )-\rho^{-1}\xi ^{-\a}\right|\leq  
	\a \eta/\rho . 
	\end{equation*}
	Hence the conclusion follows.
\end{proof}

\begin{proof}[Proof of Theorem \ref{function}]
	The proof presented here follows very closely the proof of Theorem 4.2 in \cite{DK18}.
Let us denote
\[
f(x):=e^{\a x}\E g(e^{-x}R)
\]
and
\[
\psi(x):=e^{\a x}\E g(e^{-x}R)-e^{\a x}\E g(e^{-x}AR),
\]
where $A$ and $R$ are independent, $\supp\,g\subset [1,\infty)$.
Let us define the distribution of $Z$ by 
\begin{align}\label{defZ}
\P(Z\in \cdot)=\E A^{\alpha} \Ind{\log A\in \cdot}.
\end{align} 
Then, we have for any $x>0$,
\[
f(x)=\psi(x)+e^{\a x}\E g(e^{-x}AR)=\psi(x)+\E A^\a f(x-\log A)=\psi(x)+\E f(x-Z),
\]
Iterating the above equation (see page 8 in \cite{DK18}), one arrives at
\[
f(x)=\sum_{n=0}^\infty \E \psi(x-S_n)=\int_{\R} \psi(x-z)\H(\dd z),
\]
where $\H$ is the renewal measure of $(S_n)_{n\geq 1}$ and $S_n=Z_1+\ldots+Z_n$, where $(Z_i)_i$ are independent copies of $Z$ and $S_0=0$.
Let us define 
\[
\psi_B (x):=e^{\a x}\E g(e^{-x}B) \qquad\mbox{and}\qquad \psi_0(x):=\psi(x)-\psi_B(x).
\]
Let us note that \eqref{gcond} is equivalent to 
condition $|x g'(x)  - \alpha g(x)| \leq c$ for $x\geq1$. Thus, \eqref{gcond} along with boundedness of $g$ imply that $g'$ is also bounded and so $g$ is Lipschitz continuous.

By Proposition \ref{der}, $\psi_0$ is directly Riemann integrable and so
\begin{equation*}
\lim _{x\to\infty }\int_{\R}\psi_0(x-z)\H(\dd z)=\frac{1}{\E Z}\int_{\R }\psi_0(t) \dd t<\infty.
\end{equation*}
The main contribution to the asymptotics of $f$ comes from $\int _{\R }\psi_B(x-z)\H(\dd z)$. Observe that 
	\[
	\E Z = \E A^\a \log A = \rho
	\]
and that since the law of $Z$ have the same supports as $\log A$ given $A>0$, it is also non-arithmetic. Moreover
\[
\P(Z\leq x)= \E A^\a\Ind{\log A\leq x} \leq e^{\a x}\P(A\leq e^x) = o(e^{\a x})
\]
as $x\to\infty$.
By Theorem \ref{LTH} we obtain the assertion.
\end{proof}

In the next proposition we do not need to assume that $A\geq 0$ with probability $1$ nor that the law of $R$ is the solution of the equation $R\stackrel{d}{=}AR+B$.
 We require only that the moments of $|R|$ of order strictly smaller then $\alpha$ are  finite, which is satisfied in our framework{; see \cite[Lemma~2.3.1]{BDM2016}.}
 For $0<\eps\leq 1$, we define $H^{\eps}$ to be the set of bounded functions $g$ satisfying
 \begin{equation*}
 \|g\|_{\eps}=\sup _{x,y\in \R}\frac{|g(x)-g(y)|}{|x-y|^{\eps}}<\infty.
 \end{equation*}
Clearly, due to boundedness of $g$, $H^{\eps _1}\subset H^{\eps}$ if $\eps _1\leq \eps$.  
\begin{proposition}\label{der}
	Suppose that $A,B,R$ are real valued random variables and $(A,B)$ is independent of $R$. Fix	$0<\eps < \a$, $\eps \leq  1$ 
and assume further that $\E |A|^{\a }< \infty$,  $\E\left[|A|^{\eps}|B|^{\a -\eps }\right]<\infty$, $\E |R|^{\beta }<\infty $ for every $\beta <\a$. 
	Then for every $g\in H^{\eps}$ such that $0\notin \supp \ g$ the function
	\begin{equation*}
	\psi_0(x)=e^{\a x}\E \left[g(e^{-x}(AR+B))-g(e^{-x}AR)-g(e^{-x}B)\right]
	\end{equation*}
	is directly Riemann integrable. 
\end{proposition}
\begin{proof}
	Since $\psi_0$ is continuous it is enough to prove that
	\begin{equation}
	\sum _{n\in \Z}\sup_{n\leq x<n+1} |\psi_0 (x)|< \infty .
	\end{equation}
	For $x,y\in \R $ we have
	\begin{align*}
	|g(x+y)-g(x)-g(y)|&\leq |g(x+y)-g(x)|+ |g(y)-g(0)|\leq 2\|g\|_{\eps} |y|^{\eps}. 
	\end{align*}
	Interchanging the roles of $x$ and $y$, we arrive at
	\begin{equation*}
	|g(x+y)-g(x)-g(y)|\leq 2\|g\|_{\eps}\min \{ |x|, |y|\}^\eps \Ind{\max\{|x|,|y|\}>\eta/2},
	\end{equation*}
	where $\supp\,g\subset \{ x: |x|>\eta \}$. 
Thus, for any $x\in\R$
	\begin{equation*}
	|\psi_0 (x)|\leq 2\|g\|_{\eps} e^{(\a-\eps) x} \E\left[ \min\{|B|,|AR|\}^\eps\Ind{\max\{|B|,|AR|\}>e^{x}\eta/2}\right].
	\end{equation*}
	Since $\a-\eps >0$, we have
	\begin{equation*}
	\sup_{n\leq x<n+1 }|\psi_{0} (x)|\leq 2\|g\|_{\eps} e^{(\a -\eps)(n+1)}\E\left[ \min\{|B|,|AR|\}^\eps\Ind{\max\{|B|,|AR|\}> e^n\eta }\right] 
	\end{equation*}
	and
	\begin{equation*}
	\sum_{n\in\Z}\sup_{n\leq x<n+1} |\psi_{0}(x)|\leq 2\|g\|_{\eps}\E\left[ \sum _{n=-\infty}^{n_0} e^{(\a -\eps)(n+1)}\min\{|B|,|AR|\}^\eps\right],
	\end{equation*}
	where
	\begin{equation*}
	n_0 = \left\lfloor \log\left(2\max\{|B|,|AR|\}\slash \eta\right)\right\rfloor.
	\end{equation*}
	Hence, there is a constant $C=C(\eta , \a , \eps,\|g\|_{\eps})$ such that
	\begin{align}\label{RHSX}
	\sum_{n\in\Z}\sup_{n\leq x<n+1}|\psi_{0} (x)|&\leq C\,\E\left[ \max\{|B|,|AR|\}^{\alpha-\eps}\min\{|B|,|AR|\}^{\eps}\right].
	\end{align}
	Let us first consider the case when $\alpha-2\eps> 0$. We have
	\[
	\E\left[ \max\{|B|,|AR|\}^{\alpha-\eps}\min\{|B|,|AR|\}^{\eps}\right] \leq 
	 \E\left[|B|^{\alpha-\eps}|AR|^\eps\right)] +\E\left[|AR|^{\alpha-\eps}|B|^\eps\right]. 
	\]
Since $R$ and $(A,B)$ are independent, the first term above is finite by  assumption. 
For the second term, we have
\begin{align*}
\E|R|^{\alpha-\eps}\E\left[ |A|^{\a -\eps }|B|^{\eps}\right] & =	\E|R|^{\alpha-\eps}\E\left[ |A|^\eps |A|^{\a -2\eps }|B|^{\eps}\left(\Ind{|B|\leq|A|}+\Ind{|B|>|A|}\right)\right]\\
& \leq \E|R|^{\alpha-\eps} \left(\E |A|^{\a }+\E |A|^{\eps }|B|^{\a -\eps }\right),
\end{align*}
where we have used $|A|^{\alpha-2\eps}\Ind{|B|>|A|}\leq |B|^{\alpha-2\eps}$.
On the other hand, if $\alpha-2\eps\leq0$, then we have $0<\alpha-\eps\leq\eps$ and the right hand side of \eqref{RHSX} up to a multiplicative constant $C$ is equal to
\begin{align*}
\E[ |AR|^{\eps} |B|^{\alpha-\eps}   \Ind{|B|>|AR|}]+\E[ |AR|^{\alpha-\eps} |B|^{\a-\eps} |B|^{2\eps-\a} \Ind{|B|\leq|AR|}].
\end{align*}
It is clear that both terms are finite; for the second use $|B|^{2\eps-\a} \Ind{|B|\leq|AR|}\leq |AR|^{2\eps-\a}$.
\end{proof}

\subsection{Perpetuity - second order asymptotics}\label{second}
In this section we study the second order asymptotics i.e the size of 
\[
 | x^\alpha\P(R>x)-\widetilde{L}(x)\slash \rho|
\]
when $x\to \infty $. 
For that we need more stringent assumptions on the distribution of $A$. Recall that $\widetilde{L}(x)/L(x)\to\infty$ as $x\to\infty$.
\begin{theorem}\label{perp2}
	Assume {\rm(A-1)}, {\rm(A-2)}, {\rm(B-1)}. Suppose further that there exists $\beta>0$ such that
	\begin{align}\label{condA}
	\limsup_{h\to 0^+}\sup_{a\in\R}h^{-\beta}\P(a<\log A\leq a+h)<\infty
	\end{align}
	and $\E A^\gamma<\infty$ for some $\gamma>\alpha +\alpha^2/\beta$. If the distribution of $Z$ defined by \eqref{defZ} is strongly non-lattice,
	then as $x\to\infty$,
	\begin{equation}\label{secondorder}
x^\alpha\P(R>x)=\frac{\widetilde{L}(x)}{\rho}+\frac{\E\left((AR+B)_{+}^\alpha -(AR)_{+}^\alpha-B_{+}^\alpha\right)}{\a\rho} +O(L(x))+o(1),\end{equation}
where $a_+=\max \{a,0\}$.
\end{theorem}
\begin{remark}
Depending on $L$ either the constant or $O(L(x))$ may dominate in \eqref{secondorder}. If $L(x)$ is asymptotically bounded away from zero, then  \eqref{secondorder} says that 
\[
 x^\alpha\P(R>x)=\rho^{-1}\widetilde{L}(x) + O(L(x))
\]
when $x\to \infty$.

If $L(x)\to 0$ then \eqref{secondorder} is more precise and it implies
\[
\lim_{x\to\infty} \left( x^\alpha\P(R>x)-\rho^{-1}\widetilde{L}(x)\right) = 
	\frac{\E\left((AR+B)_{+}^\alpha-(AR)_{+}^\alpha-B_{+}^\alpha\right)}{\a\rho}.
\]
\end{remark}

\begin{remark}
	In Theorem \ref{perp2} it is required that the law of $Z$ is strongly non-lattice, but it is somehow desirable to have a sufficient condition in terms of the distribution of $\log A$. It is enough to assume that the law of $\log A$ is spread-out, i.e. for some $n$ its
$n$-th convolution
has a non-zero absolutely continuous part for some $n\in\mathbb{N}$. 
If	the law of $\log A$ is spread-out then the law of $Z$ is spread-out as well. This in turn implies that the distribution of $Z$ is strongly non-lattice.
\end{remark}

We begin with the following technical Lemma.
\begin{lemma}
Under assumptions of Theorem \ref{perp2}, 
both functions
$$I_1(x)=e^{\alpha x}\P(\max\{AR,B\}\leq e^x<AR+B)$$
and 
$$I_2(x)=e^{\alpha x}\P(AR+B\leq e^x<\max\{AR,B\})$$
are $O(L(e^x))$ as $x\to\infty$.
\end{lemma}
\begin{proof}
By assumption we have $\gamma>\alpha+\alpha^2/\beta$. Take $\delta $ such that
\begin{equation}\label{cons1}
\frac{\alpha }{\gamma }< \delta <1-\frac{\alpha ^2}{\gamma \beta }
\end{equation}
and $\gamma'\in(\alpha +\alpha^2/\beta,\gamma)$ such that 
\begin{equation}\label{cons2}
\frac{\alpha ^2}{\gamma '\beta }<1-\delta .
\end{equation}
Then, we have
\begin{align*}
I_1(x)&\leq e^{\alpha x}\left( 
\P(B>e^x/2)+\P(e^{\delta x}<B\leq e^x/2, AR+B>e^x)+\P(A>e^{\alpha x/\gamma'})\right. \\
&\left.+\P(B\leq e^{\delta x},A\leq e^{\alpha x/\gamma'},AR\leq e^x\leq AR+B)
\right)\\
&= K_1+K_2+K_3+K_4.
\end{align*}
It is clear that $K_1=O(L(e^x))$. Furthermore, taking $\eta $ such that
\begin{equation}\label{const3}
\frac{\alpha ^2}{\gamma + \alpha }<\eta<\frac{\alpha\delta}{1+\delta}
\end{equation}  we obtain
$$K_2\leq e^{\alpha x}\P(ARB>e^{(1+\delta)x}/2)\leq e^{\alpha x} 2^{\alpha-\eta}\frac{\E (ARB)_+^{\alpha-\eta}}{e^{(\alpha-\eta)(1+\delta)x}}=o(e^{-s x})$$
for some $s>0$. Indeed, $\E |ARB|^{\alpha-\eta}=(\E |AB|^{\alpha-\eta})(\E |R|^{\alpha-\eta})$ and
applying H\"older inequality with $p=\gamma \slash (\alpha -\eta )$ and $q=\gamma \slash (\gamma -\alpha +\eta )$ we obtain
\begin{equation*}
\E |AB|^{\alpha-\eta}\leq (\E A^{\gamma })^{1\slash \gamma }(\E |B| ^{(\alpha -\eta )q})^{1\slash q}
\end{equation*}
and $(\alpha -\eta )q<\alpha $ in view of \eqref{const3}.
 Moreover, since $1-\gamma/\gamma'<0$ we have
$$K_3\leq e^{\alpha x} \frac{\E A^{\gamma}}{e^{\alpha \gamma x/\gamma' }}=o(e^{-s x})$$
for some $s>0$ and so $K_2$ and $K_3$ are $O(L(e^x))$ as well.
For $K_4$ define $\lambda(x)=1-e^{-(1-\delta)x}\to 1$ and recall that $\alpha/\gamma'<1$. Then, by \eqref{condA},
\begin{align*}
K_4&\leq e^{\alpha x}\P(\lambda(x)e^x <AR\leq e^x, R>\lambda(x) e^{(1-\alpha/\gamma')x}) \\
& = e^{\alpha x} \P(x-\log R+\log\lambda(x)<\log A\leq x-\log R,R>\lambda(x) e^{(1-\alpha/\gamma')x}) \\
& \leq C e^{\alpha x} \left(-\log\lambda(x)\right)^\beta \P(R>\lambda(x) e^{(1-\alpha/\gamma')x}) \\
& \sim C e^{\alpha x} e^{-\beta(1-\delta)x} \frac{\widetilde L(\lambda(x) e^{(1-\alpha/\gamma')x})}{\lambda(x)^\alpha  e^{\alpha(1-\alpha/\gamma')x}},
\end{align*}
which is $O(e^{-s x})$ for some $s>0$ in view of \eqref{cons2}.

We proceed similarly for $I_2$ writing
\begin{align*}
I_2(x)\leq e^{\a x}\left( \P(B\geq e^x)+\P(AR>e^x,-B>e^{\delta x})+\P(A>e^{\a x/\gamma'})\right.\\
\left. + \P(-B\leq e^{\delta x}, A\leq e^{\a x/\gamma'}, AR+B\leq e^x<AR)\right).
\end{align*}
Then one can show that there exists $\delta>0$ small enough to ensure that $I_2(x)=O(L(e^x))$.

\end{proof}

\begin{proof}[Proof of Theorem \ref{perp2}]
We begin the proof in the same way as in the proof of Theorem \ref{function} (see also proof of \cite[Theorem 4.2]{DK18})
but with $f(x)=e^{\alpha x}\P(R>e^x)$,
$\psi(x)=e^{\alpha x}\left( \P(AR+B>e^x)-\P(AR>e^x)\right)$,
$\psi_B(x)=e^{\alpha x}\P(B>e^x)$. Then
\begin{equation}\label{fpsi}
f(x)=\int_{\R} \psi_B(x-z)\H(\dd z)+\int_{\R} \psi_0(x-z)\H(\dd z),\end{equation}
where 
\begin{align*}
\psi_0(x)&=e^{\alpha x}\left(\P(AR+B>e^x)-\P(AR>e^x)-\P(B>e^x)\right).
\end{align*}
In view of Theorem 3.3 in \cite{DK18} we know that 
$$\int_{\R} \psi_B(x-z)\H(\dd z)=\rho^{-1}\widetilde{L}(e^x)+O(L(e^x)).$$
Hence it remains to show that 
\[\int_{\R} \psi_0(x-z)\H(\dd z)=\rho^{-1}\int_{\R}\psi_0(t)\dd t+o(1)+O(L(e^x))
\]
as $x\to\infty$.
Let us denote
\begin{align*}
I_1(x)&=e^{\alpha x}\P(\max\{AR,B\}\leq e^x<AR+B) \\
I_2(x)&=e^{\alpha x}\P(AR+B\leq e^x<\max\{AR,B\}) \\
I_3(x)&=e^{\alpha x}\P(\min\{AR,B\}>e^x)
\end{align*}
so that 
\[\psi_0(x)=I_1(x)-I_2(x)-I_3(x).\]
In the proof of Theorem 4.2 in \cite{DK18} we have already shown (under weaker assumptions) that 
\[
\int_{\R}I_3(x-z)\H(\dd z)=\frac{\E\min\{AR,B\}_+^\alpha}{\alpha\rho}+o(1)
\]
and that $\E\min\{AR,B\}_+^\alpha<\infty$.
By the preceding Lemma we know that $I_i(x)=O(L(e^x))$ for $i=1,2$ and this implies that as $x\to\infty$,
$$\int_{(-\infty,0]} I_i(x-z)\H(\dd z)=O(L(e^x)),\qquad i=1,2.$$ 
Indeed,
consider $\int_{(-\infty,0]} \frac{L(e^{x-z})}{L(e^x)}\H(\dd z)$. For any $\delta>0$, the integrand is bounded by $c e^{-\delta z}$ for some $c>1$ by Potter's bound \eqref{potter}.
Combining this with \eqref{Hleft} and Lebesgue's Dominated Convergence Theorem we
conclude that 
\begin{align}\label{LH0}
\int_{(-\infty,0]} L(e^{x-z})\H(\dd z)\sim L(e^x)H(0).
\end{align}

Observe that there exists
 $\bt>0$ such that
	\begin{align}\label{cond1}
	\limsup_{h\to 0^+}\sup_{a\geq 0}h^{- \bt}\P(a<Z\leq a+h)<\infty .
	\end{align}
Indeed, let $p=\frac{\a +\eps }{\a } $, $q=\frac{\a +\eps}{\eps}$ with $\alpha+\eps\leq \gamma$. Then
	$$\P(a<Z\leq a+h)=\E A^{\a }\Ind {a<\log A \leq a+h}\leq \big (\E A^{\a +\eps }\big )^{1\slash p}\big (\P  (a<\log A \leq a+h)\big )^{1\slash q}.$$
Hence
	$$
h^{-\beta \slash q}\P(a<Z\leq a+h)\leq \big (\E A^{\a +\eps }\big )^{1\slash p}\big (h ^{-\beta }\P  (a<\log A \leq a+h)\big )^{1\slash q}
$$
and \eqref{cond1} follows by \eqref{condA}.
In view of \eqref{cond} we have the following easy result for $x>u$ and $d>u$,
\begin{align}\label{tech}\begin{split}
\int_{((x-d)_+,x-u]} e^{\alpha(x-z)}\H(\dd z)&\leq e^{\alpha d}\H((x-d)_+,x-u]) \\
& \leq c\,  e^{\alpha d} \max \{ (x-u-(x-d)_+)^{\ba}, x-u-(x-d)_+\}\\
& \leq c\,  e^{\alpha d} \max \{ (d-u)^{\ba}, d-u\} 
\end{split}\end{align}
for some $\ba>0$, where, 
the first inequality follows from monotonicity of the integrand and the second one by {Lemma \ref{FROST}}.

Moreover, notice that for $0<\lambda \leq 1$ and all $x>0$ one has
$\log(1+x)\leq \lambda^{-1} x^\lambda$.
Let us note that on the event $\{AR+B>\max\{AR,B\}\}$, both $AR$ and $B$ are positive and, on the space restricted to this event, random variables $U=\log\max\{AR,B\}$ and $D=\log(AR+B)$ are well defined.  Then, by \eqref{tech}
\begin{align*}
&\int_{(0,\infty)}I_1(x-z)\H(\dd z)=\E \int_{(0,\infty)} e^{\alpha(x-z)}\Ind{\max\{AR,B\}\leq e^{x-z}<AR+B}\H(\dd z) \\
& = \E \int_{(x-D,x-U]\cap(0,\infty)} e^{\alpha(x-z)}\H(\dd z)\Ind{D>U} \\
&\leq c\,\E (AR+B)^\alpha ( (D-U)^{\ba } + (D-U)) \Ind{D>U}.
\end{align*}
For the first term above we have
\begin{align*}
&c\,\E (AR+B)^\alpha  (D-U)^{\ba } \Ind{D>U}\\
& = c\,\E (AR+B)^\alpha \left(\log\left(1+\frac{\min\{AR,B\}}{\max\{AR,B\}}\right)\right)^{\ba} \Ind{AR+B>\max\{AR,B\}} \\
& \leq \frac{c}{\lambda^{\ba}} \E (AR+B)^\alpha\frac{\min\{AR,B\}^{\lambda \ba}}{\max\{AR,B\}^{\lambda \ba}} 
\Ind{AR+B>\max\{AR,B\}}\\
&\leq 2^\alpha \frac{c}{\lambda^{\ba}} \E \max\{AR,B\}^{\alpha-\lambda \ba} \min\{AR,B\}^{\lambda \ba}\Ind{AR+B>\max\{AR,B\}}\\
& \leq 2^\alpha \frac{c}{\lambda^{\ba}} \big (\E (AR)^{\alpha-\lambda \ba } B^{\lambda \ba}\Ind{\min\{AR,B\}=B>0} + \E B^{\alpha-\lambda \ba} (AR)^{\lambda\ba}\Ind{\min\{AR,B\}=AR>0}\big ) <\infty 
\end{align*}
provided $\ba \lambda <\a $. An analogous calculation shows that
$\E (AR+B)^\alpha  (D-U) \Ind{D>U}<\infty $ and so $\int_{(x-D,x-U]\cap(0,\infty)} e^{\alpha(x-z)}\H(\dd z)\Ind{D>U}$ is dominated  by an integrable random variable which does not depend on $x$.
Thus, by Lebesgue's Dominated Convergence Theorem we have
\begin{align*}
\lim_{x\to\infty} \int_{(0,\infty)}I_1(x-z)\H(\dd z)=\E \lim_{x\to\infty}\int_{(0,x-U]} e^{\alpha(x-z)}\Ind{x-z<D}\H(\dd z)\Ind{D>U}
\end{align*}
and for $d>u$ as $x\to\infty$,
$$e^{\a u}\int_{(0,x-u]} e^{\alpha(x-u-z)}\Ind{x-u-z<d-u}\H(\dd z)\to \rho^{-1} e^{\a u} \int_0^{\infty} e^{\alpha t} \Ind{t<d-u}\dd t,$$
where we have used the Key Renewal Theorem since  $x\mapsto e^{\alpha x}\Ind{[0,d-u)}$  is dRi (it has compact support, is bounded and a.e. continuous). Thus
$$\lim_{x\to\infty} \int_{(0,\infty)}I_1(x-z)\H(\dd z)=\frac{\E ((AR+B)^\a-\max\{AR,B\}^\a)\Ind{AR+B>\max\{AR,B\}}}{\alpha \rho}.$$

We proceed similarly with $I_2$. With $D=\log\max\{AR,B\}$ and $U=\log(AR+B)$ (analogously as before, $D$ and $U$ are well defined on the events $\{\max\{AR,B\}>0\}$ and $\{AR+B>0\}$, respectively), we have
\begin{align*}
&\int_{(0,\infty)} I_2(x-z)\H(\dd z)=\E \int_{(0,\infty)} e^{\alpha(x-z)}\Ind{AR+B\leq e^{x-z}<\max\{AR,B\}}\H(\dd z) \\
& \leq \E \int_{(x-D,x-U]\cap (0,\infty)}e^{\alpha(x-z)}\H(\dd z) \Ind{\max\{AR,B\}>AR+B\geq 2^{-1}\max\{AR,B\}>0}\\
&+ \E \int_{(x-D,\infty)\cap(0,\infty)}e^{\alpha(x-z)}\H(\dd z) \Ind{ AR+B\leq 2^{-1}\max\{AR,B\}, \max\{AR,B\}>0}.
\end{align*}
and by \eqref{tech}
\begin{multline*}
\E \int_{(x-D,x-U]\cap(0,\infty)}e^{\alpha(x-z)}\H(\dd z) \Ind{\max\{AR,B\}>AR+B\geq 2^{-1}\max\{AR,B\}>0}\\
\leq \E \max\{AR,B\}^\alpha \Big (\left(\log\max\{AR,B\}-\log(AR+B)\right)^{\ba} 
\\+ \left(\log\max\{AR,B\}-\log(AR+B)\right)\Big)\Ind{\max\{AR,B\}>AR+B\geq 2^{-1}\max\{AR,B\}>0}.%\Ind{ AR+B\leq 2^{-1}\max\{AR,B\}, \max\{AR,B\}>0}.
\end{multline*}
Again, as before we do calculations for the term with $\ba$. It is bounded by 
\begin{align*}
&c\, \E \max\{AR,B\}^\alpha \left(\log \big (1+\frac{-\min\{AR,B\}}{AR+B}\big )\right)^{\ba} \Ind{\max\{AR,B\}>AR+B\geq 2^{-1}\max\{AR,B\}>0}\\
&\leq\frac{c}{\lambda^{\ba}} \E \max\{AR,B\}^\alpha \left(\frac{-\min\{AR,B\}}{AR+B}\right)^{\lambda\ba} \Ind{\max\{AR,B\}>AR+B\geq 2^{-1}\max\{AR,B\}>0}\\
&\leq 2^{\lambda\beta}\frac{c}{\lambda^{\ba}} \E \max\{AR,B\}^\alpha \left(\frac{|\min\{AR,B\}|}{\max \{AR,B\}}\right)^{\lambda \ba} \Ind{\max\{AR,B\}>AR+B\geq 2^{-1}\max\{AR,B\}>0}\\
&\leq2^{\lambda\beta}\frac{c}{\lambda^{\ba}} \E \max\{AR,B\}^{\alpha - \lambda\ba } |\min\{AR,B\}|^{\lambda \ba} <\infty
\end{align*}
as before. 
The second term equals
\begin{align*}
&\alpha\,\E \max\{AR,B\}^\alpha \int_0^\infty e^{-\alpha t}\H((x-D,x-D+t])\dd t\,\Ind{AR+B\leq 2^{-1}\max\{AR,B\},\max\{AR,B\}>0 }\\
&\leq c\,\E \max\{AR,B\}_+^\alpha \Ind{AR+B\leq 2^{-1}\max\{AR,B\},\max\{AR,B\}>0 }.
\end{align*}
Now, since $\min \{ AR,B\}\leq 0$ and
$$
AR+B=\max\{AR,B\}+\min \{ AR,B\}\leq \frac{1}{2}\max\{AR,B\}$$
we have
$$|\min \{ AR,B\}|\geq \frac{1}{2}\max\{AR,B\}$$
and
\begin{align*}
\E \max\{AR,B\}^\alpha \Ind{|\min \{ AR,B\}|\geq 2^{-1}\max\{AR,B\}>0}&\leq \E B^\alpha \Ind{B>0, AR<0, 1\leq 2\frac{|A R|}{B}}+\E (AR)^\alpha \Ind{AR>0, B<0, 1\leq 2\frac{|B|}{A R}}\\
&\leq 2^{\eta }\left(\E |B|^\alpha \left(\frac{|A R|}{|B|}\right)^{\eta} +\E |AR|^\alpha \left(\frac{|B|}{|A R|}\right)^{\eta }\right)\\
&\leq 2^{\eta }\left(\E |R|^{\eta } 
\E |B|^{\alpha-\eta }A^\eta 
+\E|R|^{\a-\eta} \E A^{\alpha-\eta}|B|^\eta \right)<\infty. 
\end{align*}
Similarly as before, Lebesgue's Dominated Convergence Theorem implies that as $x\to\infty$,
$$\int_{(0,\infty)}I_2(x-z)\H(\dd z)\to \frac{\E(\max\{AR,B\}^\alpha-(AR+B)_+^\a)_+}{\alpha\rho}$$
and so as $x\to\infty$, after straightforward simplification,
$$\int_\R \psi_0(x-z)\H(\dd z)= \frac{1}{\a\rho}  \E\left( (AR+B)_+^\alpha-(AR)_+^\alpha-B_+^\alpha\right)+O(L(e^x))+o(1).$$

\end{proof}

\section{Perpetuities with general $A$}\label{general}
Now we are going to consider perpetuities with $A$ attaining negative values as well. More precisely, we assume that $\P (A<0)>0$, possibly with
$\P (A\leq 0)=1$. Our aim is to reduce the general case to the one already solved: non-negative $A$. We propose a unified approach to perpetuities, which applies beyond our particular assumptions.

Assume that $\E\log|A|<0$ and $\E\log^+ |B|<\infty$. Then the stochastic equation $R\stackrel{d}{=}AR+B$ with $(A,B)$ and $R$ independent has a unique solution, or equivalently, that $R_n=A_nR_{n-1}+B_n$, $n\geq 1$, converges in distribution to $R$ for any $R_0$ independent of $(A_n,B_n)_{n\geq 1}$, where  $(A_n,B_n)_{n\geq 1}$ is a sequence of independent copies of the pair $(A,B)$. 

Define the filtration $\mathbb{F}=\{\mathcal{F}_n\colon n\geq 1\}$, where $\mathcal{F}_n=\sigma((A_k,B_k)_{k=1}^n)$. 
Following \cite[Lemma 1.2]{ver79}, for any 
stopping time $N$ (with respect to $\mathbb{F}$) which is finite with probability one, $R$ satisfies 
\begin{equation}\label{newrec}
R\stackrel{d}{=}A_1\ldots A_N R+R_N^\ast,\qquad \mbox{ $R$ and $(A_1\ldots A_N,R_N^\ast)$ are independent},\end{equation}
where $R_n^\ast=B_1+A_1 B_2+\ldots+A_1\ldots A_{n-1}B_n$ for $n\geq 1$. 
For $n\geq 1$ we write $\Pi_n=A_1\cdot\ldots\cdot A_n$ and $\Pi_0=1$.
Let $N:=\inf\{n\in\mathbb{N}\colon \Pi_n\geq0\}$. Then, $N$ is a stopping time with respect to $\mathbb{F}$ and $N$ is finite with probability $1$. Indeed, if $\P (A\leq 0)=1 $ then $N=2$. If $\P (A>0)>0$ then $N=\infty $ if and only if $A_1<0$ and for every $n\geq 2$, $A_n>0$ which means that for every $n$
\begin{equation}
\P (N=\infty )\leq \P (A<0)\P (A>0)^{n-1}\to 0, \quad \mbox{as } n\to \infty.
\end{equation}
Let now $\P (A<0)>0$ and $A_+=A\Ind {A\geq 0}$, $A_-=-A\Ind {A< 0}$.

Since $\{N\geq k\}=\{A_1<0,A_2>0,\ldots,A_{k-1}>0\}$ for $k\geq 2$ we have 
$$R_N^\ast=\sum_{k=1}^\infty \Ind{N\geq k}\Pi_{k-1}B_k =B_1-(A_1)_{-}\left(\sum_{k=2}^\infty (A_2)_+\cdots (A_{k-1})_+B_k\right).$$
Let us denote the expression in brackets by $S$. Then, $S$ is independent of $((A_1)_-,B_1)$ and it is the unique solution to
	\begin{align}\label{perpS}
	S\stackrel{d}{=}A_+S+B,\qquad \mbox{where}\ S\mbox{ and }(A_+,B)\mbox{ are independent}.
	\end{align}
Summing up, we obtain
\begin{lemma}\label{LemSigned}
	Assume that $\P(A<0)>0$ with $\E\log|A|<0$ and $\E\log^+|B|<\infty$. Let $R$ be the solution to
	$$R\stackrel{d}{=}AR+B,\qquad  \mbox{$R\mbox{ and }(A,B)$ are independent}.$$
	Then $R$ is also a solution to \eqref{newrec}, where 
$$R_N^\ast\stackrel{d}{=}(-A_-)S+B,\qquad \mbox{$S\mbox{ and }(A_-,B)$ are independent}$$
and $S$ satisfies \eqref{perpS}. 
\end{lemma}
Thanks to the above lemma, we can reduce the case of signed $A$ to the case on non-negative $A$. The properties of $\Pi_N$ and $R_N^\ast$ will be inherited by the properties of the original $(A,B)$.

The main result of this section is
\begin{theorem}\label{signedA}
Suppose that
\begin{itemize}
	\item[\rm (sA-1)] $\P (A<0)>0$, $\E\log |A|<0$,
	\item[\rm (sA-2)] there exists $\alpha>0$ such that $\E |A|^{\alpha}=1$, $\rho=\E |A|^{\alpha}\log |A|<\infty$,
	\item[\rm (sA-3)] the distribution of $\log |A|$ given $|A|>0$ is non-arithmetic,
   \item[\rm (sA-4)] there exists $\eps >0$ such that $\E |A|^{\a +\eps } <\infty$, 
    \item[\rm (sB-1)] $$\P(B>t)\sim p\, t^{-\alpha}L(t),\qquad \P(B<-t)\sim q\, t^{-\alpha}L(t), \qquad p+q=1,$$
    \item[\rm (sB-2)] $\E |B|^\alpha=\infty$.
\end{itemize}
Then 
\begin{equation}
x^{\a }\P (R>x)\sim \frac{\widetilde L(x)}{2\rho },\quad x^{\a }\P (R<-x)\sim \frac{\widetilde L(x)}{2\rho }.
\end{equation}
\end{theorem}

The proof relies on Lemma \ref{LemSigned}. The tail asymptotics of $S$ follows from \cite{Gre94} as it is explained below in the proof of Theorem \ref{signedB}. 
In view of \eqref{newrec} to conclude Theorem \ref{signedA} it remains to 
prove that $\Pi _N$ and $R_N^{\ast}$ satisfy assumptions of Theorem \ref{perpet}.
First we will prove that $\Pi_N$ inherits its properties from $A$. The following result is strongly inspired by \cite[(9.11)-(9.13)]{Gol91} (see also \cite[Lemma 4.12]{Als}). For completeness, the proof is included below.
\begin{theorem}
	\begin{itemize}
		\item[{\rm(i)}] 	If the law of $\log |A|$ given $A\neq0$ is non-arithmetic (spread-out), then the law of $\log \Pi_N$ given $\Pi_N> 0$ is non-arithmetic (spread-out),
		\item[{\rm(ii)}] If $\E|A|^\alpha=1$ and $\E|A|^{\alpha+\eps}<\infty$ for some $\eps>0$ then
		there exists $\bar\eps>0$ such that $\E \Pi_N^{\alpha+\bar\eps}<\infty$,
		\item[{\rm(iii)}] If $\E|A|^\alpha=1$ then $\E \Pi_N^{\alpha}=1$ and $\E \Pi_N^{\alpha}\log\Pi_N=2 \E|A|\log|A|$.
	\end{itemize}
\end{theorem}
\begin{proof}
If $P(A\leq 0)=1$ then $\Pi _N=A_1A_2$ and the law of $\log \Pi _N$ given
$\log \Pi _N>0$ is $\P_<\ast\P_<$, where $\P_<$ is the law of $\P_{\log|A|\big|A<0}$. $\P_<\ast\P_<$ is non-arithmetic or spread out respectively if so is $\P_<$. Also the remaining of the above statements are clear in this case so for the rest of the proof we assume that $P(A>0)>0$.
	\begin{itemize}
		\item[\rm(i)] 
		Denote by $\P_>$ and $\P_<$ the laws of $\P_{\log A\big|A>0}$ and $\P_{\log|A|\big|A<0}$, respectively. Set $p=\P(A>0)$ and $q=\P(A<0)$. 
		By \cite[(9.11)]{Gol91}, we have
		$$\P_{\log\Pi_N|\Pi_N>0}=\frac{1}{\P(\Pi_N>0)}\left(p\, \P_>+q^2\P_<^{\ast2}\sum_{n=0}^\infty p^n\P_>^{\ast n}\right).$$

If $p\,\P _>+q\,\P _<$ is spread out then there are $k,l\geq 0$ such that 
$\P _>^{\ast k}\ast \P _<^{\ast l}$ has a non zero absolutely continuous 
component. Hence $\P_>\ast \P_<^{\ast 2}$ is spread out and the mixture of measures, one of which is spread-out is spread-out as well. 

If $p\,\P _>+q\,\P _<$ is non-arithmetic then the supports of $\P _>$ and $\P _<^{\ast 2}$ generate a dense subgroup of $\R $ (see the argument below \cite[(9.13)]{Gol91}).  
Thus, we conclude that $\P_{\log\Pi_N|\Pi_N>0}$ is non-arithmetic.
		 
		\item[\rm(ii)] Let $\mu_+^{(\varepsilon)}:=\E A^{\alpha+\varepsilon}\Ind{A\geq 0}$. Since the function $\varepsilon\mapsto \mu_+^{(\varepsilon)}$ is continuous and $\mu_+^{(0)}<1$, then there exists $\varepsilon_1>0$ such that $\mu_+^{(\varepsilon_1)}<1$.
		
		Then, we have
		\begin{align*}
		\E\Pi_N^{\alpha+\varepsilon_1}&= \E A_1^{\alpha+\varepsilon_1}\Ind{A_1\geq 0}+\sum_{n=2}^\infty \E \Pi_n^{\alpha+\varepsilon_1}\Ind{A_1<0,A_2>0,\ldots,A_{n-1}>0,A_{n}\leq 0} \\
		&=\mu_+^{(\varepsilon_1)}+\left(\E|A|^{\alpha+\varepsilon_1}\Ind{A<0}\right)^2 \sum_{n=2}^\infty (\mu_+^{(\varepsilon_1)})^{n-2}<\infty.
		\end{align*}
		
		\item[\rm(iii)]
		Define a measure $\mathbb{Q}_n$ on $(\Omega,\mathcal{F}_n)$ by
		$$\mathbb{Q}_n(S):=\E |\Pi_{n}|^\alpha \Ind{S},\quad S\in \mathcal{F}_n,\quad n\geq 0.$$
		Let $\mathcal{F}_\infty$ be the smallest $\sigma-$field containing all $\mathcal{F}_n$. The sequence of measures $\mathbb{Q}_n$ is consistent, thus by Kolmogorov theorem there exists a unique measure $\mathbb{Q}$ on $\mathcal{F}_\infty$ such that $\mathbb{Q}(S)=\mathbb{Q}_n(S)$ for $S\in\mathcal{F}_n$.
		Note that $(A_n)_{n\geq 1}$ are i.i.d. also under $\Q$.
		We have
		$$\mu_+:=\Q(N=1)=\Q(A_1\geq 0)=\E |A|^\alpha \Ind{A>0}=\mu _+^{(0)}$$ 
		and for any $k>1$,
		$$\Q(N=k)=\Q(A_1<0,A_2>0,\ldots,A_{k-}>0, A_k\leq 0) =(1-\mu_+)^2\mu_+^{k-2}.$$
		Hence $\E_{\Q}N=2$, where $\E_{\Q}$ is the expectation with respect to $\Q$.
		
		Since $\mathcal{F}_N\subset\mathcal{F}_\infty$, for any $S\in\mathcal{F}_N$ we have
		\begin{align*}
		\mathbb{Q}(S)&=\sum_{n=1}^\infty \mathbb{Q}(S\cap\{N=n\}) =\sum_{n=1}^\infty \E |\Pi_{n}|^\alpha \Ind{S\cap\{N=n\}}\\
		&=\sum_{n=1}^\infty \E \Pi_{N}^\alpha \Ind{S\cap\{N=n\}}=\E \Pi_{N}^\alpha \Ind{S}.
		\end{align*}
		Putting $S=\Omega$ we obtain that $\E \Pi_{N}^\alpha=1$. 
		Further, since $\Pi_N$ is $\mathcal{F}_N$ measurable, we have
		$$\E\Pi_{N}^\alpha\log\Pi_{N}=\E_{\mathbb{Q}} \log\Pi_N=\E_{\mathbb{Q}}\left(\sum_{n=1}^N \log |A_n|\right)=\E_{\mathbb{Q}}N\cdot \E_{\mathbb{Q}}\log|A_1|=2\E |A|^{\alpha} \log |A|,
		$$
		where the Wald's identity was used.
	\end{itemize}
\end{proof}

Secondly we show that the tails of $R_N^\ast$ behave like $\P(|B|>x)$. 
Let now $\P (A>0)>0$ and $A_+=A\,\Ind {A\geq 0}$, $A_-=-A\,\Ind {A< 0}$.

\begin{theorem}\label{signedB}
		Assume additionally that
		$$\P(B>t)\sim p\, t^{-\alpha}L(t),\qquad \P(B<-t)\sim q\, t^{-\alpha}L(t), \qquad p+q=1$$
		and $\E |A|^{\alpha+\eps}<\infty$ for some $\eps>0$.
		If $\mu_+=\E A^{\alpha}\Ind{A>0}<1$, 
		then 
		\begin{align}\label{tailS}
		\P(S>t)\sim\frac{1}{1-\mu_+}\P(B>t),\qquad
		\P(S<-t)\sim\frac{1}{1-\mu_+}\P(B<-t),
		\end{align}
		and
		\begin{align}\label{tailsRN}
		\P(R_N^\ast>t)\sim \P(|B|>t)\sim\P(R_N^\ast<-t).
		\end{align}
\end{theorem}
\begin{proof}
		Tail asymptotic of $S$ follow from the application of \cite[Theorem 3]{Gre94} to $(M,Q,R)=(A_+,B,S)$. We have $\E |M|^\alpha=\E A^{\alpha}\Ind{A>0}<1$ and $\E |M|^{\alpha+\varepsilon}\leq\E |A|^{\alpha+\varepsilon}<\infty$ by the assumption.
		
		Tail asymptotics of $R_N^\ast$ then follow from \cite[Lemma 4]{Gre94}, since $R_N^\ast\stackrel{d}{=}B+A^{-}S$. Here $(M,Q,Y)=(A_-,B,S)$ and $\E |M|^\alpha=\E A^{\alpha}\Ind{A<0}<1$ and $\E|M|^{\alpha+\varepsilon}$ is finite as above. 
		One easily checks that $\P(R_N^\ast>t)\sim \P(|B|>t)$. To obtain $\P(R_N^\ast<-t)\sim \P(|B|>t)$ we apply the above argument to $-R\stackrel{d}{=}A(-R)-B$.
\end{proof}

\section{Proof of Theorem \ref{LTH}}\label{proofs}

	First we prove that
	\begin{equation}\label{negligable}
	\lim _{x\to \infty}\widetilde L(e^x)^{-1}\int _{(-\infty,0] }e^{\a (x-z)}\E g(e^{z-x}B)\H(\dd z)=0
	\end{equation}
Since $g$ is bounded and its support is contained in $[1,\infty)$, there exists a constant $c$ such that $g(x)\leq c\,\Ind{x> 1-\eps}$ for any $\eps>0$.
Thus, with $c_\eps = c(1-\eps)^{-1}$ we have $e^{\a (x-z)}\E g(e^{z-x}B)\leq c\, e^{\a (x-z)}\P(B> (1-\eps)e^{x-z})=c_\eps L((1-\eps)e^{x-z})$ and therefore
\begin{align*}
\int _{(-\infty,0] }e^{\a (x-z)}\E g(e^{z-x}B)\H(\dd z)&\leq c_\eps \int_{(-\infty,0]} L((1-\eps)e^{x-z})\H(\dd z) \\
& \sim c_\eps L(e^x)H(0)=o(\widetilde L(e^x)).
\end{align*}
by \eqref{LH0}.

For the main part we have
	\begin{align*}
	\int_{(0,\infty) }e^{\a (x-z)}\E g(e^{-(x-z)}B)\H(\dd z)&=\int_{(0,\infty) }e^{\a (x-z)}\E g(e^{-(x-z)}B)\Ind {\{B>e^x\}}\H(\dd z)\\
	&+\int_{(0,\infty) }e^{\a (x-z)}\E g(e^{-(x-z)}B)\Ind {\{0<B\leq e^x\}}\H(\dd z)= I_1(x)+I_2(x).
	\end{align*}
The first term is easily seen to be $O(L(e^x))$. Indeed, observe that the integral $\int_{(0,\infty) }e^{-\a z}\H(\dd z)=\int_0^\infty \alpha e^{-\alpha x}\H\left((0,x] \right)\dd x$ is finite by \eqref{Hineq1}. 
Bounding $g$ by an indicator as before, we have
\begin{align*}
I_1(x)&\leq c_\eps \int_{(0,\infty) }e^{\a (x-z)}\P(B>(1-\eps)e^{x-z},B>e^x)\H(\dd z)
	\\
	&= c_\eps L(e^x)\int_{(0,\infty) }e^{-\a z}\H(\dd z).
\end{align*}
Let us decompose $I_2(x)$ in the following way
\begin{align}\label{I2decomp}
I_2(x)=\E \int _0^{\infty }e^{\a (x-z)} g(e^{-(x-z)}B)\frac{\dd z}{\E Z} \,\Ind{0<B\leq e^x} +
\E \int _0^{\infty }e^{\a (x-z)} g(e^{-(x-z)}B)\dd\left(H(z)-\frac{z}{\E Z}\right) \,\Ind{0<B\leq e^x}
\end{align}
The first term above is
\[
\a (\E Z)^{-1} \widetilde L(e^x) \int _0^{\infty }g(r)r^{-\a -1}\dd r+O(L(e^x))
\]
and it constitutes the main ingredient in \eqref{eq1}.
To see this, change the variable $r=e^{-(x-z)}B$, to obtain
	\begin{align*}
	(\E Z)^{-1} \E B_+^{\a}\Ind{B\leq e^x}\ \int _0^{\infty } g(r)\ r^{-\a -1}\ \dd r, 
	\end{align*}
	by the fact that $\supp\,g\subset[1,\infty)$.
	But \eqref{Bax} gives us that
	$\E B_+^{\a}\Ind{B\leq e^x}=\a \widetilde L(e^x)-L(e^x).$
	It remains to prove that the second term in \eqref{I2decomp} is $o(\widetilde L(e^x))$.
Let us denote 
\[R(z)=H(z)-\frac{z}{\E Z}.\]
	The equality $g(1)=0$ is a consequence of $\mathrm{supp}\,g\subset[1,\infty)$ and differentiability of $g$. Since $\lim _{z\to \infty }e^{-\alpha z}R(z)=0$, after integrating by parts we see that
	\begin{align*}
	\E \int_{(x-\log B,\infty)}e^{\a (x-z)} &g(e^{-(x-z)}B)\dd R(z)\Ind{0<B\leq e^x}\\
	&=-\E \int _{x-\log B}^\infty\frac{\dd}{\dd z}\Big (e^{\a (x-z)} g(e^{-(x-z)}B)\Big )R(z)\dd z\Ind{0<B\leq e^x}\\
	&=-\E B^\alpha \int _{0}^\infty\frac{\dd}{\dd t}\Big (e^{-\a t} g(e^{t})\Big )R(t+x-\log B)\dd t\Ind{0<B\leq e^x},\\
	\end{align*}
	where we have substituted $t=z-x+\log B$.
	Moreover, notice that
	\begin{equation*}
 \int _{0}^\infty\frac{\dd}{\dd t}\Big (e^{-\a t} g(e^{t})\Big )\dd t=0
	\end{equation*}
	and so
	\begin{align*}
	&\E B^\alpha\Ind{0<B\leq e^x} \int _{0}^\infty\frac{\dd}{\dd t}\Big (e^{-\a t} g(e^{t})\Big )R(t+x-\log B)\dd t \\
	&=\E B^\alpha\Ind{0<B\leq e^x} \int _{0}^\infty\frac{\dd}{\dd t}\Big (e^{-\a t} g(e^{t})\Big )\left( R(t+x-\log B)-R(x-\log B)\right)\dd t
	\end{align*}
	By the assumption \eqref{gcond}, there exists a constant $C$ such that for all $t>0$, 
	\begin{equation*}
	\Big | \frac{\dd}{\dd t}\big (e^{-\a t} g(e^{t})\big )\Big |\leq Ce^{-\a t},
	\end{equation*} 
	so it amounts to estimate
	\begin{align}
	\E B^\alpha\Ind{0<B\leq e^x} \int _{0}^\infty e^{-\alpha t}\left| R(t+x-\log B)-R(x-\log B)\right|\dd t\label{eqR}
	\end{align}
Define 
$$J(x):=\frac{ \E B^\alpha\Ind{0<B\leq e^x} \int _{0}^\infty e^{-\alpha t}\left| R(t+x-\log B)-R(x-\log B)\right|\dd t}{\E B^\alpha\Ind{0<B\leq e^x}}.$$ 
We will show that $J(x)\to 0$, and since the denominator equals $\alpha \widetilde{L}(e^x)$ this will be the end of the proof.

Define the law of $C_x$ by
$$\P(C_x\in\cdot)=\frac{ \E B_+^\alpha \Ind{B\leq e^x, B\in\cdot}}{\E B_+^\alpha \Ind{B\leq e^x}}.$$
Note that $\P(0<C_x\leq e^x)=1$.
Thus, $J(x)$ may be rewritten as
$$\E \int _{0}^\infty e^{-\alpha t}\left| R(t+x-\log C_x)-R(x-\log C_x)\right|\dd t.$$
Since for any positive $x$ and $t$, $|R(t+x)-R(x)|=|\H((x,x+t])-\frac{t}{\E Z}|\leq c\,t+b$ for some $c,b>0$, we have
$$\lim_{x\to\infty} J(x)=\int_0^\infty e^{-\alpha t} \lim_{x\to\infty}\E \left| R(t+x-\log C_x)-R(x-\log C_x)\right| \dd t.$$
Moreover, $x-C_x$ converges to infinity in probability, as $x\to\infty$.
Indeed, for any $N>0$ we have
$$\P(x-\log C_x\geq N)=\P(C_x\leq e^{x-N})=\frac{\E B_+^\alpha \Ind{B\leq e^{x-N}}}{\E B_+^\alpha \Ind{B\leq e^x}}=\frac{\widetilde L(e^{x-N})}{\widetilde L(e^x)}\to 1,$$
because $\widetilde{L}$ is slowly varying.
Since, $\left|R(t+x)-R(x)\right|\to 0$ as $x\to\infty$, we infer that 
\begin{equation}\label{rozn}
\left|(R(t+x-\log C_x)-R(x-\log C_x)\right|
\end{equation}
converges to $0$ in probability, as $x\to\infty$.
But \eqref{rozn} is bounded in $x$, thus the convergence holds also in $L_1$ and we may finally conclude that
$$\lim_{x\to\infty}J(x)=0,$$
which completes the proof of \eqref{eq1}.

If additionally $\E \exp(\eps Z)<\infty$ for some $\eps>0$ and the law of $Z$ is strongly non-lattice then \eqref{eqR} is bounded by $CL(\exp(x))$ which was proved in \cite{DK18} - see the end of the proof of Theorem 3.3 there just before the references.

%%%%%%%%%%%%%%%%%%%%%%%%%%%%%%%%%%%%%%%%%%%%%%%%%%%%%%%%%%%%%
%%%%%%%%%%%%%%%%%%%%%%%%%%%%%%%%%%%%%%%%%%%%%%%%%%%%%%%%%%%%%
%%%%%%%%%%%%%%%%%%%%%%%%%%%%%%%%%%%%%%%%%%%%%%%%%%%%%%%%%%%%%

\section{Appendix}
Suppose that $\E |B|^{\beta}<\infty $ for any $\beta <\a $ and that there is $\eps >0$ such that $\E |A|^{\a +\eps }<\infty $. Then by H\"older inequality
we may conclude that for every $\eta <\a $, $\E |B|^{\a -\eta  }|A|^{\eta }<\infty$. However, if the tail of $B$ exhibits some more regularity, a weaker condition implies the same conclusion.

Suppose that $x^{\a }\P(|B|>x)\leq L(x)$, where $L$ is a slowly varying function bounded away from $0$ and $\infty $ on any compact subset of $(0,\infty )$. 
% Then it is natural to suspect that $\E |A|^{\a }S(|A|)<\infty $ for some slowly varying function will be sufficient. Indeed, 
Let $W$ be a non-decreasing function such that
\begin{equation*}
W(x)\geq C\max \{ L(x), \log(x)\}\quad \mbox{for} \ x\geq 0.
\end{equation*}
For instance $W(x)=\max \{ \sup _{0<w\leq x}L(w), \log(x)\}$ or $W(x)=\max \{\widetilde{L}(x), \log(x)\}$ will do.

\begin{lemma}\label{lemmaAB} Assume that $W$ is as above, $\eta <\a$, $D>\frac{2\a }{\eta }-1$  and 
\begin{equation*}
\E |A|^{\a }W(|A|)^D<\infty . 
\end{equation*}
Then
\begin{equation*}
\E |B|^{\a -\eta  }|A|^{\eta }<\infty . 
\end{equation*}
\end{lemma}
\begin{proof}
Since $\E |A|^{\a }<\infty $ and $\E |B|^{\a -\eta  }<\infty $, it is enough to prove that for a fixed $C_0$
\begin{equation*}
\E |B|^{\a -\eta  }|A|^{\eta }\Ind {|B|>|A| }\Ind {|A|\geq C_0 }<\infty . 
\end{equation*}
We choose $\beta >0$ and $\g $ such that
\begin{equation}\label{gamma}
1-\frac{\eta }{\a }<\g <1.
\end{equation}
 For $m\geq k$ consider the sets
\begin{equation*}
S_{k,m}=\big \{ e^kW(e^k)^{\b }<|A|\leq  e^{k+1}W(e^{k+1})^{\b },
e^mW(e^m)^{\b }<|B|\leq  e^{m+1}W(e^{m+1})^{\b }\big \}
\end{equation*}
Let $C_0=e^{k_0}W(e^{k_0})^{\b }$, where $k_0$ is such that $W(k_0)\geq 1$. Then %, for sufficiently large $k_0$,
\begin{equation*}
\E |B|^{\a -\eta  }|A|^{\eta }\Ind {|B|>|A| }\Ind {|A|\geq C_0 } \leq C\sum _{k\geq k_0}
\sum _{m\geq k}e^{m(\a -\eta )+k\eta }W(e^m)^{\b (\a -\eta )}W(e^k)^{\b \eta }\P (S_{k,m}),
\end{equation*}
\begin{equation*}
\P (S_{k,m})\leq \P \big ( |B|> e^mW(e^m)^{\b }\big )^{\g} \P \big (|A|> e^kW(e^k)^{\b }\big ) ^{1-\g} 
\end{equation*}
and
\begin{equation*}
\P \big ( |B|> e^mW(e^m)^{\b }\big )\leq e^{-\a m}W(e^m)^{-\a \b}L\big (
e^{m}W(e^m)^{ \b}\big ).
\end{equation*}
Let $\delta >0$. By the Potter bounds \eqref{potter}, since $W(e^m)\geq 1$ for $m\geq k_0$
$$
L\big (e^{m}W(e^m)^{ \b}\big )\leq C L\big (
e^{m}\big )W(e^m)^{\delta \b}\leq C W\big (
e^{m}\big )W(e^m)^{\delta \b}.$$
Hence 
\begin{equation*}
\P \big ( |B|> e^mW(e^m)^{\b }\big )
\leq C e^{-\a m}W(e^m)^{1-\a \b+\b \d}.
\end{equation*}
Further,
\begin{align*}
\P \big ( |A|> e^kW(e^k)^{\b }\big )&\leq \big (\E |A|^{\a }W(|A|)^D\big ) e^{-\a k}W(e^k)^{-\a \b}W\big (
e^{k}W(e^k)^{ \b}\big )^{-D}\\
&\leq C e^{-\a k}W(e^k)^{-D-\a \b},
\end{align*}
because  $W$ is non-decreasing and $W(e^k)\geq 1$.
Therefore, we have
\begin{align*}
\E |B|^{\a -\eta  }|A|^{\eta }&\Ind {|B|>|A| }\Ind {|A|\geq C_0 }\\
& \leq C\sum _{k\geq k_0}
\sum _{m\geq k}e^{(m-k)(\a -\eta -\a \g)}W(e^m)^{\b (\a -\eta )+\g (1-\a \b+\b \d) }W(e^k)^{\b \eta -\a \b (1-\g )-D(1-\g )} . 
\end{align*}
Notice that in view of \eqref{gamma}
$$\a -\eta -\a \g <0.$$
In order to sum up over $\beta $, we choose $\beta $ such that
$$
\beta (\a -\eta -\a \g)<-\g .$$
Finally, we take $\delta $ sufficiently small to ensure
$$
\beta (\a -\eta -\a \g +\g \delta)<-\g.$$
Then 
$$
\sum _{m\geq k}e^{(m-k)(\a -\eta -\a \g)}W(e^m)^{\b (\a -\eta )+\g (1-\a \b+\b \d) }\leq \sum _{m\geq 0}e^{m(\a -\eta -\a \g)}<\infty .$$
Hence
\begin{align*}
\E |B|^{\a -\eta  }|A|^{\eta }&\Ind {|B|>|A| }\Ind {|A|\geq C_0 }\\
& \leq C\sum _{k\geq k_0}
W(e^k)^{\b (\eta -\a +\a \g )-D(1-\g )}\\
&\leq C\sum _{k\geq k_0}
k^{\b (\eta -\a +\a \g )-D(1-\g )}<\infty, 
\end{align*}
Finally, we need to guarantee that
\begin{equation}\label{exponent}
\beta (\eta -\a +\a \g)-D(1-\g )<-1.\end{equation}
Suppose that 
$\beta (\a -\eta -\a \g)=-\g -\xi$ for some $\xi >0$. Then \eqref{exponent} becomes
\begin{equation}\label{exponent1}
D>\frac{1+\g +\xi}{1-\g }.\end{equation}
and we may minimize $D$ by an appropriate choice of $\gamma $. Notice that if $\xi =0$ and $\g = 1-\eta \slash \a $ the right hand side of \eqref{exponent1} becomes $2\a \slash \eta -1$. Since $\g $ may be arbitrary close to $1-\eta \slash \a $ and $\xi $ arbitrary close to $0$, $D>2\a \slash \eta -1$ will do.
\end{proof}

\subsection*{Acknowledgements}
Ewa Damek was partially supported by the NCN Grant UMO-2014/15/B/ST1/00060.
Bartosz Ko{\l}odziejek was partially supported by the NCN Grant UMO-2015/19/D/ST1/03107.

\bibliographystyle{plain}

\begin{thebibliography}{21}
\providecommand{\natexlab}[1]{#1}
\providecommand{\url}[1]{\texttt{#1}}
\expandafter\ifx\csname urlstyle\endcsname\relax
  \providecommand{\doi}[1]{doi: #1}\else
  \providecommand{\doi}{doi: \begingroup \urlstyle{rm}\Url}\fi

\bibitem{Als}
G.~Alsmeyer.
\newblock \emph{Random Recursive Equations and Their Distributional Fixed
  Points}.
\newblock 2015.
\newblock available on-line at
  http://wwwmath.uni-muenster.de/statistik/lehre/SS15/StRek/StochRekgl.pdf.

\bibitem{Alsm2014}
G.~Alsmeyer.
\newblock On the stationary tail index of iterated random Lipschitz functions.
\newblock \emph{Stochastic Process. Appl.}, 126\penalty0(1):\penalty0 209--233,
  2016.

%\bibitem{Ara06}
%V.~F. Araman and P.~W. Glynn.
%\newblock Tail asymptotics for the maximum of perturbed random walk.
%\newblock \emph{Ann. Appl. Probab.}, 16\penalty0 (3):\penalty0 1411--1431,
%  2006.

%\bibitem{AR15}
%J.~Arista and V.~Rivero.
%\newblock Implicit renewal theory for exponential functionals of L\'evy processes.
%\newblock \emph{arXiv:1510.01809}, pages 1--38, 2015.

\bibitem{AC}
L.~Arnold and H.~Crauel.
\newblock Iterated function systems and multiplicative ergodic theory, in {Diffusion processes and related problems in analysis, {V}ol. {II} ({C}harlotte, {NC}, 1990)}
\newblock \emph{Birkh\"{a}user}, Boston,  pages 283--305, 1992

%@incollection {MR1187996,
%	AUTHOR = {Arnold, Ludwig and Crauel, Hans},
%	TITLE = {Iterated function systems and multiplicative ergodic theory},
%	BOOKTITLE = {Diffusion processes and related problems in analysis, {V}ol.
%		{II} ({C}harlotte, {NC}, 1990)},
%	SERIES = {Progr. Probab.},
%	VOLUME = {27},
%	PUBLISHER = {Birkh\"{a}user Boston, Boston, MA},
%}

\bibitem{AMN78}
	K.~B. Athreya and D.~McDonald and P.~Ney.
	\newblock Limit theorems for semi-{M}arkov processes and renewal theory for {M}arkov chains.
	\newblock \emph{Ann. Probab.}, 6\penalty0 (5):\penalty0 788--797,
	1978.

\bibitem{BGT89}
N.~H. Bingham, C.~M. Goldie, and J.~L. Teugels.
\newblock \emph{Regular variation}, volume~27 of \emph{Encyclopedia of
  Mathematics and its Applications}.
\newblock Cambridge University Press, Cambridge, 1989.

\bibitem{Black53}
D.~Blackwell.
\newblock Extension of a renewal theorem.
\newblock \emph{Pacific J. Math.}, 3:\penalty0 315--320, 1953.

\bibitem{BB}
  S.~Brofferio, D.~Buraczewski.
\newblock  On unbounded invariant measures of stochastic dynamical systems.
\newblock {\em  Ann. Probab.}, 43\penalty0 (3):\penalty0 1456--1492, 2015.


\bibitem{BD}
 D.~Buraczewski, E.~Damek. 
\newblock  A simple proof of heavy tail estimates for affine type Lipschitz recursions.  
\newblock \emph {Stochastic Process. Appl.}, 127:\penalty0 657-668, 2017.


\bibitem{BDGHU2009}
D.~Buraczewski, E.~Damek, Y.~Guivarch, A.~Hulanicki, and R.~Urban.
\newblock Tail-homo\-gene\-ity of stationary measures for some multidimensional
  stochastic recursions.
\newblock \emph{Probab. Theory Related Fields}, 145\penalty0 (3):\penalty0
  385--420, 2009.

\bibitem{BDM2016}
D.~Buraczewski, E.~Damek, and T.~Mikosch.
\newblock \emph{Stochastic Models with Power-Law Tails. The Equation $X=AX+B$.}
\newblock Springer Series in Operations Research and Financial Engineering.
  Springer International Publishing, Switzerland, 2016.

\bibitem{BDP}
 D.~Buraczewski,  E.~Damek, T.~Przebinda.
\newblock On the rate of convergence in the Kesten renewal theorem.
\newblock \emph{Electron. J. Probab.} 20 \penalty0 (22):\penalty0 1--35, 2015.
\bibitem{BDIM}
D.~Buraczewski, P.~Dyszewski, A.~Iksanov, and A. Marynych.
\newblock On perpetuities with gamma-like tails.
\newblock \emph{J. Appl. Probab.} 55\penalty0 (2):\penalty0 368--389, 2018.


\bibitem{DD} 
E.~Damek and P.~Dyszewski,
\newblock Iterated random functions and regularly varying tails.  
\newblock \emph{J. Difference Equ. Appl.} 24\penalty0 (9):\penalty0 1503-1520.
  
\bibitem{DK18}
E.~Damek and B.~Ko{\l}odziejek.
\newblock A renewal theorem and supremum of a perturbed random walk.
\newblock \emph{Electron. Commun. Probab.}, 23:\penalty0 Paper No. 82:\penalty0 1--13, 2018.

\bibitem{DF}
P.~Diaconis, D.~Freedman.
\newblock Iterated random functions.
\newblock \emph {SIAM Rev.}, 41\penalty0 (1):\penalty0 45--76, 1999
(electronic);

\bibitem{Du}
M.~Duflo 
\newblock \emph {Random Iterative Systems.}
\newblock{Springer Verlag, New York, 1997.}

\bibitem{Dy}
P.~Dyszewski.
\newblock Iterated random functions and slowly varying tails 
\newblock \emph {Stochastic Process. Appl.} 
 126\penalty0 (2):\penalty0 392-413, 2016.


\bibitem{Elton}
 J.~H. Elton.
\newblock A multiplicative ergodic theorem for {L}ipschitz maps,
\newblock {\em Stochastic Process. Appl.}, 34:\penalty0 39--47, 1990.

%\bibitem{ESZ09}
%N.~Enriquez, C.~Sabot, and O.~Zindy.
%\newblock A probabilistic representation of constants in {K}esten's renewal
%  theorem.
%\newblock \emph{Probab. Theory Related Fields}, 144\penalty0 (3-4):\penalty0
%  581--613, 2009.

\bibitem{Gol91}
C.~M. Goldie.
\newblock Implicit renewal theory and tails of solutions of random equations.
\newblock \emph{Ann. Appl. Probab.}, 1\penalty0 (1):\penalty0 126--166, 1991.

\bibitem{GG96}
C.~M. Goldie and R.~Gr{\"u}bel.
\newblock Perpetuities with thin tails.
\newblock \emph{Adv. in Appl. Probab.}, 28\penalty0 (2):\penalty0 463--480,
  1996.

\bibitem{Gre94}
D.~R. Grey.
\newblock Regular variation in the tail behaviour of solutions of random
  difference equations.
\newblock \emph{Ann. Appl. Probab.}, 4\penalty0 (1):\penalty0 169--183, 1994.

\bibitem{Gri75}
A.~K. Grincevi\v{c}ius.
\newblock One limit distribution for a random walk on lines,
\newblock \emph{Lith. Math. J.}, 15\penalty0 (4):\penalty0 580--589, 1975.


\bibitem{HH}
  H.~Hennion, L.~Herv\'e. 
\newblock Central limit theorems for iterated random Lipschitz mappings.
\newblock {\em Ann. Probab.}, 32(3A),\penalty0 1934-1984, 2004.

\bibitem{HitWes09}
P.~Hitczenko and J.~Weso{\l}owski.
\newblock Perpetuities with thin tails revisited.
\newblock \emph{Ann. Appl. Probab.}, 19\penalty0 (6):\penalty0 2080--2101,
  2009.

%\bibitem{Ren11}
%P.~Hitczenko and J.~Weso{\l}owski.
%\newblock Renorming divergent perpetuities.
%\newblock \emph{Bernoulli}, 17\penalty0 (3):\penalty0 880--894, 2011.

\bibitem{Iks}
A.~Iksanov.
\newblock \emph{Renewal theory for perturbed random walks and similar
	processes.}
\newblock Probability and its Applications. 
Birkh\"auser/Springer, Cham, 2017.


%\bibitem{IksPi14}
%A.~Iksanov and A.~Pilipenko.
%\newblock On the maximum of a perturbed random walk.
%\newblock \emph{Statist. Probab. Lett.}, 92:\penalty0 168--172, 2014.

\bibitem{Kes73}
H.~Kesten.
\newblock Random difference equations and renewal theory for products of random matrices.
\newblock \emph{Acta Math.}, 131:\penalty0 207--248, 1973.


\bibitem{KM96}
H.~Kesten and R.~A. Maller.
	\newblock Two renewal theorems for general random walks tending to
	infinity.
	\newblock \emph{Probab. Theory Related Fields}, 106\penalty0 (1):\penalty0 1--38,
	1996.

\bibitem{Kevei2016}
P.~Kevei.
\newblock A note on the {K}esten-{G}rincevi\v cius-{G}oldie theorem.
\newblock \emph{Electron. Commun. Probab.}, 21:\penalty0 Paper No. 51, 12,
  2016.

\bibitem{BK16}
B.~Ko{\l}odziejek.
\newblock The left tail of renewal measure.
\newblock \emph{Statist. Probab. Lett.}, 129:\penalty0 306--310, 2017.

\bibitem{BKc16}
B.~Ko{\l}odziejek.
\newblock Logarithmic tails of sums of products of positive random variables bounded by one.
\newblock \emph{Ann. Appl. Probab.}, 27\penalty0 (2):\penalty0 1171--1189, 2017.

\bibitem{BK18}
B.~Ko{\l}odziejek.
\newblock On perpetuities with light tails.
\newblock \emph{Adv. in Appl. Probab.}, 50\penalty0 (4):\penalty0 1119--1154, 2018.


\bibitem{Mi}
M.~Mirek.
\newblock Heavy tail phenomenon and convergence to stable laws for iterated Lipschitz maps.
\newblock \emph{Probab. Theory Related Fields}, 151\penalty0 :(3-4),\penalty0 705--734, 2011.

%\bibitem{Palm07}
%Z.~Palmowski and B.~Zwart.
%\newblock Tail asymptotics of the supremum of a regenerative process.
%\newblock \emph{J. Appl. Probab.}, 44\penalty0 (2):\penalty0 349--365, 2007.


\bibitem{Root}
H.~Rootz\'{e}n.
\newblock Extreme value theory for moving average processes.
\newblock \emph{Ann. Probab.}, 14\penalty0 (2):\penalty0 612--652, 1986.

\bibitem{Stone65}
C.~Stone.
\newblock On moment generating functions and renewal theory.
\newblock \emph{Ann. Math. Statist.}, 36:\penalty0 1298--1301, 1965.

\bibitem{ver79}
W.~Vervaat.
\newblock On a stochastic difference equation and a representation of
  nonnegative infinitely divisible random variables.
\newblock \emph{Adv. in Appl. Probab.}, 11\penalty0 (4):\penalty0 750--783,
  1979.

%\bibitem{Wan14}
%Y.~Wang.
%\newblock Convergence to the maximum process of a fractional {B}rownian motion
%  with shot noise.
%\newblock \emph{Statist. Probab. Lett.}, 90:\penalty0 33--41, 2014.

\end{thebibliography}

\def\polhk#1{\setbox0=\hbox{#1}{\ooalign{\hidewidth
  \lower1.5ex\hbox{`}\hidewidth\crcr\unhbox0}}}

\end{document}